\documentclass[12pt,reqno]{amsart}

\usepackage{amsmath,amssymb,amsthm,mathrsfs}
\usepackage{graphicx,cite,times}
\usepackage{cases}
\usepackage{color}
\usepackage{appendix}
\usepackage{enumerate}

\usepackage[colorlinks, linkcolor=blue, citecolor=blue]{hyperref}
\usepackage{cite}
\usepackage{graphicx}

\newtheorem{theo}{Theorem}[section]

\newtheorem{lemm}[theo]{Lemma}

\newtheorem{rema}[theo]{Remark}
\newtheorem{defi}[theo]{Definition}

\numberwithin{equation}{section}

\begin{document}

\title[Limiting absorption principle of Helmholtz equation under periodic structure]{Limiting absorption principle  of Helmholtz equation with sign changing coefficients under periodic structure}

\author{Wenjing Zhang}
\address{School of Mathematics and Statistics, and Center for Mathematics
and Interdisciplinary Sciences, Northeast Normal University, Changchun,
Jilin 130024, China}
\email{zhangwj034@nenu.edu.cn}

\author{Yu Chen}
\address{Department of Mathematics, Northeastern University, Shenyang, Liaoning 110004, China }
\email{chenyu@amss.ac.cn}

\author{Yixian Gao}
\address{School of Mathematics and Statistics, Center for Mathematics and Interdisciplinary Sciences, Northeast Normal University, Changchun, Jilin 130024, China}
\email{gaoyx643@nenu.edu.cn}

\thanks{The research of YG was supported by NSFC grant 12371187  and Science and Technology Development Plan Project of Jilin Province 20240101006JJ. The research of YC was partially supported by NSFC grant 12371117, the FRF for the Central Universities N2305007 and NSF of Liaoning Province ZX20240488.}

\subjclass[2010]{35J15, 35B34, 35B35}

\keywords{Helmholtz euqation, negative index materials, limiting absorption principle, periodic structure}

\begin{abstract}
Negative refractive index materials have attracted significant research attention due to their unique electromagnetic response characteristics.
In this paper, we employ the  complementing boundary condition to establish rigorous a priori estimates for the  Helmholtz equation,   from which the limiting absorption principle is analytically derived. Within this mathematical framework, we conclusively establish the well-posedness of the electromagnetic  transmission problem at the interface between  conventional materials and negative refractive index materials in two-dimensional periodic structures.
\end{abstract}

\maketitle

\section{introduction}
Negative index materials (NIMs) are a type of artificially synthesized metamaterial that exhibits exceptional electromagnetic properties,
 distinct from those of conventional materials.  These include simultaneous negative electric permittivity and magnetic permeability, as well as negative refraction characteristics.
  In 1968, Veselago investigated the theoretical aspects of materials possessing both negative permittivity and permeability \cite{Veselago1968The}. Subsequent advances in the 1990s revealed that periodically arranged thin metallic wires demonstrate negative permittivity at specific frequencies \cite{pendry1996extremely}, while metallic split-ring resonators with periodic configurations exhibit negative permeability \cite{WOS:000083406900003}.
  A landmark achievement occurred in 2001 when Shelby et al., building on Pendry's foundational work, experimentally observed negative refraction phenomena \cite{WOS:000167995200039}, a discovery that sparked broad scientific interest. Recent developments include Sharma et al.'s investigation of ultrathin metamaterials capable of simultaneous negative permittivity and permeability within defined frequency ranges, highlighting its potential for controlling electromagnetic waves \cite{WOS:001222725500002}. Furthermore, Berrington et al. explored how to achieve negative refractive
index in dielectric crystals containing rare-earth ions. This research implies that we can now explore
negative refractive index effects at optical wavelengths in three-dimensional natural media \cite{WOS:001159239300001}.

Banerjee and Nehmetallah formulated  partial differential equations for wave propagation in NIMs based on dispersion relations and verified the linear characteristics of wave propagation within these negative refractive index media \cite{MR2445138}. Aylo systematically investigated the fundamental properties of electromagnetic wave propagation in NIMs, along with various tools for understanding and characterizing these materials\cite{aylo2010wave}. Their research has deepened our understanding of NIMs and enriched the theoretical foundation for their practical applications. NIMs have a wide range of practical applications.
For example, NIM-based lenses achieve subwavelength-resolution imaging, offering transformative potential in medical diagnostics  \cite{WOS:000796891400003,MR3458160,WOS:001180030400005}.
The negative refractive index property of NIMs enables the fabrication of superlenses that reduce the size and weight of optical components, thereby improving system performance \cite{PhysRevLett.85.3966,MR3871415}. Moreover, NIMs enable controlled electromagnetic wavefront manipulation, laying the foundation for broadband cloaking devices \cite{ammari2013spectral,MR4128423,MR4405196,MR3947973,MR3990650}. Additionally, NIMs have applications in antennas \cite{WOS:001063223800029}. Pandey et al. investigated the wave propagation characteristics of one-dimensional photonic crystals in both TE and TM modes, providing design guidelines for their applications in areas such as antennas \cite{WOS:000970513800038}.

 The simultaneous  presence of negative permittivity and permeability in NIMs introduces sign reversals in the coefficients of governing equations when analyzing interactions between traditional materials and NIMs. These sign changes invalidate classical methods for solving elliptic equations, thereby necessitating  the development of novel mathematical approaches such as integral equation techniques, variational formulations, and the $T$-coercivity method. Costabel and Stephan  established the  well-posedness of the Helmholtz equation with sign-changing coefficients by applying  the Fredholm framework under the critical ratio condition   ${\varepsilon^{(1)}}/{\varepsilon^{(2)}}\neq-1$, where $\varepsilon^{(1)}$ and $\varepsilon^{(2)}$  denote the dielectric permittivity  of regions $D$ and $\mathbb{R}^n\setminus\overline{D}(n\geqslant2)$, respectively\cite{MR782799}. Bonnet-Ben Dhia, Ciarlet, and Zw\"{o}lf demonstrated the well-posedness of the Helmholtz equation with sign-changing coefficients via variational methods\cite{10.1016/j.cam.2006.01.046}, later extending their framework to the  three-dimensional Maxwell systems\cite{MR2446403}. The $T$-coercivity method introduced by Bonnet-Ben Dhia et al. in \cite{WOS:000278570900038}, provides an alternative framework for addressing the well-posedness of Helmholtz equations with sign-varying coefficients \cite{MR3062923,dhia2012t,MR4277850,MR3534861}, offering improved computational efficiency.   Furthermore, Nguyen established the well-posedness of such equations via a priori estimates derived from the Cauchy problem \cite{MR3515306}. In a recent advancement, Hu and Kirsch employed the limiting absorption principle to impose constraint conditions on scattering problem solutions, thereby proving the well-posedness of scattering phenomena associated with periodic curves \cite{hu2024directinversetimeharmonicscattering}.

This work establishes the well-posedness  of the  Helmholtz equation involving NIMs and traditional materials in  periodic structures.
The analysis faces two principal challenges. First, conventional approaches for elliptic equations become inapplicable due to the sign reversal of the coefficients.
Second, uniqueness of solutions for the Helmholtz equation is typically ensured by imposing conditions at infinity; the Sommerfeld radiation condition  serves as a natural choice.
However, in the periodic structures studied here, this condition is no longer applicable.
 Faced with these challenges and inspired by \cite{MR3515306}, we employed complementing boundary conditions to conduct a priori estimates for the reduced system \eqref{c1}.
 Unlike the approach in \cite{MR3515306}, we introduced a transparent boundary condition to convert the original problem on an unbounded domain into one defined on a bounded domain. Within this bounded domain, we further investigated the well-posedness of  the system \eqref{c2} with sign-changing coefficients, thereby simplifying the computational process.  Additionally, due to the sign reversal of coefficients in the Helmholtz equation defined on the periodic structure of system \eqref{c2}, we constructed mappings to transform this system into an elliptic system on the half-plane   $\{x\in\mathbb{R}^2:\:x_2\geqslant 0\}$, subject to complementing boundary conditions.
  Subsequently,  by utilizing the definition and properties of the difference operator $D_1^h$ (where  $h$ is a sufficiently small positive constant), we derived a boundary $H^2$-regularity result.
  By combining these results  with the  complementing boundary conditions for a priori estimates of the Cauchy problem, we proved the validity of the limiting absorption principle. The limiting absorption principle lies in introducing an infinitesimal absorption term to seek the unique solution of the equation, making this method particularly suitable for solving complex problems involving NIMs. Consequently,  we proved the existence and uniqueness of the solution to the Helmholtz equation in \eqref{c2} via the limiting absorption principle. Furthermore, based on the uniqueness of the limit, we also derived the stability estimates. These results collectively demonstrate the well-posedness of the Helmholtz equation with sign-changing coefficients on periodic structures.

The organization of this paper is as follows: In Section \ref{2}, we delve into the mathematical model underlying the problem, introduce  necessary notations used throughout the paper, and present the main  theorem of this paper. Section \ref{cbc} explains the definition of the complementing boundary condition, illustrates its application to the considered boundary value problem, and presents the $H^2$ regularity result. Section \ref{3} introduces key lemmas involved in the proof. The proof of the main theorem is discussed in detail in Section \ref{4}.

\section{problem formulation}\label{2}
Since the structure and
medium are assumed to be periodic in the $x_1$-direction, there exists a constant $\Lambda>0$ such that
the profile of grating $S$  in one period can be represented as
\begin{align*}
S:=\{x\in\mathbb{R}^2:0<x_1<\Lambda,\:x_2=f\left(x_1\right)\},
\end{align*}
where $f\in C^{2}(\mathbb R)$ satisfies
\[
f\left(x_1+n\Lambda\right)=f\left(x_1\right) \quad  \text {for all}~ n\in\mathbb{Z},\]
and $f$ is additionally assumed to be non-constant.

Assume  the region above the grating is occupied by a homogeneous medium,
 where both the electric permittivity and magnetic permeability are positive constants, i.e.,
\begin{align*}
\varepsilon_1>0,\quad\mu_1>0.
\end{align*}
Below the grating is a negative-index medium characterized by
\begin{align*}
\varepsilon_2<0,\quad\mu_2<0,
\end{align*}
where both material parameters are negative.
Since the Helmholtz equation  becomes ill-posed in $H^1$ when the contrast ratio  $\varepsilon_1/\varepsilon_2=-1$ ( due to the change of sign in the coefficient \cite{MR3200087}), we assume $\varepsilon_1/\varepsilon_2\neq -1$  throughout this paper.

Denote the problem geometry by the following regions:
\[
\begin{aligned}
\tilde{\Omega}_1 &:= \left\{ x \in \mathbb{R}^2 : 0 < x_1 < \Lambda,\ x_2 > f(x_1) \right\}, \\
\Omega_2        &:= \left\{ x \in \mathbb{R}^2 : 0 < x_1 < \Lambda,\ 0 < x_2 < f(x_1) \right\}, \\
S_0             &:= \left\{ x \in \mathbb{R}^2 : x_1 = 0,\ x_2 > 0 \right\}, \\
S_\Lambda       &:= \left\{ x \in \mathbb{R}^2 : x_1 = \Lambda,\ x_2 > 0 \right\}, \\
\Gamma          &:= \left\{ x \in \mathbb{R}^2 : 0 < x_1 < \Lambda,\ x_2 = 0 \right\},
\end{aligned}
\]
and let $\tilde{\Omega} := \tilde{\Omega}_1 \cup \Omega_2 \cup S$. The geometry is illustrated in Figure \ref{fig1}.

Throughout this paper, the symbol $C$ denotes a generic constant whose value may vary depending on the context, and $\nu$ represents the unit normal vector.
\begin{figure}[h]
    \centering
    \includegraphics[width=0.45\textwidth]{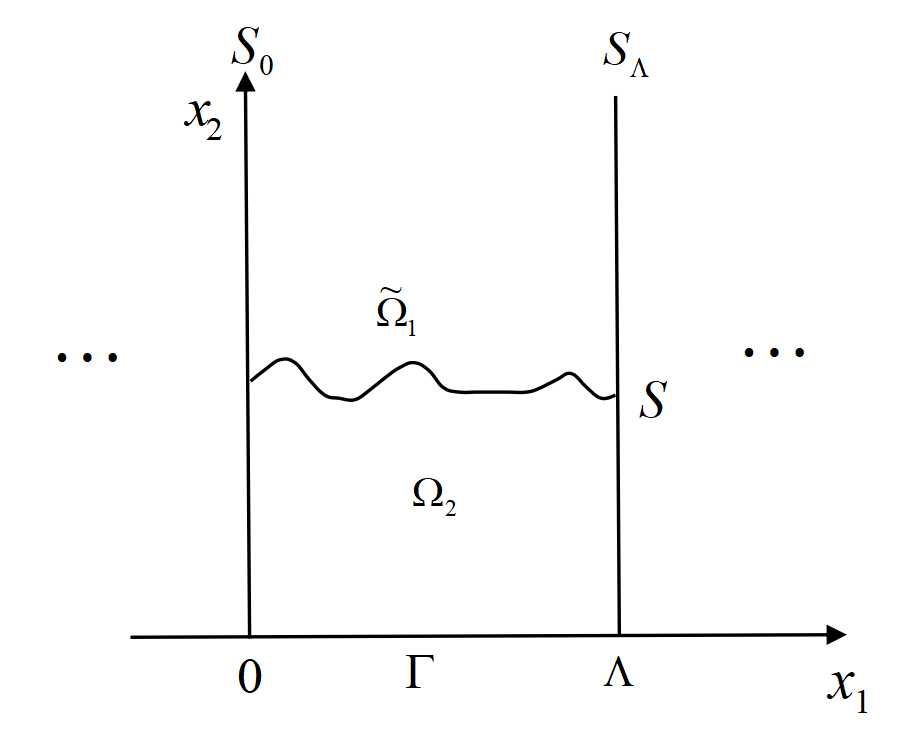}
    \caption{ Problem geometry}
    \label{fig1}
\end{figure}

Consider an incident plane wave \( u^i \) impinging on the structure from above, expressed explicitly as
\begin{align*}
u^i(x)=e^{{\mathrm i}\left(\alpha x_1-\beta x_2\right)},\quad x=(x_1, x_2)\in\mathbb{R}^2,
\end{align*}
where the parameters are defined as
\[\alpha=\kappa_1 \sin\theta, \quad  \beta=\kappa_1 \cos\theta, \]
 with $\theta\in\left(-\pi/2,\pi/2\right)$ denoting the incident angle, $\omega>0$  the angular frequency, and  $\kappa_1=\omega\sqrt{\varepsilon_1\mu_1}$  the wave number in the medium above the grating.  The boundary $\Gamma$ is modeled as a perfect electric conductor, implying a homogeneous Neumann boundary condition on $\Gamma$.

We study the well-posedness of the following boundary value problem:
\begin{align}\label{ab}
\begin{cases}
\nabla\cdot\left(\varepsilon_0^{-1}\nabla u_0\right)+\omega^2\mu u_0=0&\quad \text{in}\ \tilde{\Omega},\\
\partial_{x_2}u_0=0&\quad \text{on}\ \Gamma,
\end{cases}
\end{align}
where the piecewise-constant coefficients are defined as
\begin{align*}
\varepsilon_0:=\begin{cases}
\varepsilon_1&\quad \text{in}\ \tilde{\Omega}_1,\\
\varepsilon_2&\quad \text{in}\ \Omega_2,
\end{cases}\quad
\mu:=\begin{cases}
\mu_1&\quad \text{in}\ \tilde{\Omega}_1,\\
\mu_2&\quad \text{in}\ \Omega_2,
\end{cases}
\end{align*}
and  $\kappa_2^2:=\omega^2\varepsilon_2\mu_2$. The following transmission conditions are imposed on the  grating profile $S$:
\begin{align*}
u_0|_{\tilde{\Omega}_1}=u_0|_{\Omega_2},\quad\varepsilon_1^{-1}\nabla u_0|_{\tilde{\Omega}_1}\cdot\nu=\varepsilon_2^{-1}\nabla u_0|_{\Omega_2}\cdot\nu.
\end{align*}

To demonstrate the  existence of the boundary value  problem \eqref{ab},
we employ the limiting absorption principle by analyzing the  following  problem
\begin{align}\label{1a}
\begin{cases}
\nabla\cdot\left(\varepsilon_\sigma^{-1}\nabla u_\sigma\right)+\omega^2\mu u_\sigma+{\rm i}\sigma u_\sigma=0&\quad \text{in}\ \tilde{\Omega},\\
\partial_{x_2}u_\sigma=0&\quad \text{on}\ \Gamma,
\end{cases}
\end{align}
where $\sigma$ is a piecewise function defined as
\begin{align*}
\sigma=\sigma(x)=\begin{cases}
0\quad&\text{in}\ \tilde{\Omega}_1,\\
\sigma>0\quad&\text{in}\ \Omega_2.
\end{cases}
\end{align*}
The modified permittivity $\varepsilon_\sigma$ is given by
\begin{align*}
\varepsilon_\sigma:=\begin{cases}
\varepsilon_1&\quad \text{in}\ \tilde{\Omega}_1,\\
\varepsilon_2+\dfrac{{\rm i}\sigma}{\omega}&\quad \text{in}\ \Omega_2.
\end{cases}
\end{align*}
Transmission conditions on the grating profile $S$ are imposed as
\begin{align*}
u_\sigma \big|_{\tilde{\Omega}_1}=u_\sigma \big|_{\Omega_2},\quad\varepsilon_1^{-1}\nabla u_\sigma \big|_{\tilde{\Omega}_1}\cdot\nu=\left(\varepsilon_2+\dfrac{{\rm i}\sigma}{\omega}\right)^{-1}\nabla u_\sigma \big|_{\Omega_2}\cdot\nu.
\end{align*}

Since the solutions $u_0$ and $u_\sigma$ of the boundary value problems \eqref{ab} and \eqref{1a} are quasi-periodic functions, they inherit the quasi-periodicity on the vertical boundaries $S_0$ and $S_\Lambda$.  Additionally, these solutions satisfy  the bounded outgoing condition as $ x_2 \to \infty $.  The total field $u_\sigma$ above the grating can be decomposed into the incident field $ u^i$ and the diffracted field $ u^d $, i.e.,
\[
u_\sigma=u^d+u^i,\quad x_2>f(x_1).
\]

In order to investigate the above problem, we apply a transparent boundary condition ( see  \cite[Chapter 3]{bao2022maxwell}) to transform the unbounded domain into a bounded region.
  For problem  \eqref{1a}, the quasi-periodicity of  the diffracted field $u^d$ allows the Fourier series expansion:
\begin{align*}
u^d(x_1,x_2)=\sum_{n\in\mathbb{Z}}u_n^d(x_2)e^{{\rm i}\alpha_nx_1},\quad\alpha_n=\alpha+n\dfrac{2\pi}{\Lambda}
\end{align*}
with   Fourier coefficients defined by
\begin{align*}
u_n^d(x_2)=\dfrac{1}{\Lambda}\int_0^\Lambda u^d(x_1,x_2)e^{-{\rm i}\alpha_nx_1} {\rm d} x_1.
\end{align*}
Within the domain  $\tilde{\Omega}_1$, the diffracted field  $u^d(x_1,x_2)$ satisfies the  time-harmonic Helmholtz equation
\begin{align*}
\Delta u^d+k_1^2u^d=0\quad \text{in}\ \tilde{\Omega}_1.
\end{align*}
The propagation constants are defined as
\begin{align*}
\beta_n:=\begin{cases}
\left(\kappa_1^2-\alpha_n^2\right)^{\frac{1}{2}},\quad&|\alpha_n|<\kappa_1,\\
{\rm i} \left(\alpha_n^2-\kappa_1^2\right)^{\frac{1}{2}},\quad&|\alpha_n|>\kappa_1.
\end{cases}
\end{align*}
Through direct calculation, we obtain the Rayleigh expansion of $u^d(x_1,x_2)$ as follows
\begin{align*}
u^d(x_1,x_2)=\sum\limits_{n\in\mathbb{Z}}u_n^d(h_1)e^{{\rm i}\left(\alpha_nx_1+\beta_n(x_2-h_1)\right)},\quad x_2>h_1,
\end{align*}
where   $h_1$ is a  constant satisfying   $h_1>\max\limits_{0<x_1<\Lambda}f(x_1)$ , and the modal coefficients are
\begin{align*}
u_n^d(h_1)=\dfrac{1}{\Lambda}\int_0^\Lambda u^d(x_1,h_1)e^{- {\rm i}\alpha_nx_1}{\rm d}x_1.
\end{align*}

For a quasi-periodic function $\omega (x_1)$ with the Fourier expansion  $\omega(x_1)=\sum\limits_{n\in\mathbb{Z}}\omega_n e^{\mathrm {i} \alpha_nx_1}$,  where the Fourier coefficients are
$
\omega_n=\dfrac{1}{\Lambda}\int_0^\Lambda\omega(x_1)e^{-{\rm i} \alpha_nx_1} {\rm d}x_1,
$
the Dirichlet-to-Neumann (DtN) operator $T$  is  defined  by
\begin{align}\label{a0}
\left(T\omega\right)(x_1):=\sum\limits_{n\in\mathbb{Z}}{\rm i} \beta_n\omega_ne^{{\rm i}\alpha_nx_1}.
\end{align}
Consequently, the diffracted field
$u^{d}$ satisfies the boundary condition
\begin{align*}
\partial_{x_2}u^d=Tu^d\quad \text{on}\ \Gamma_0,
\end{align*}
with  $\Gamma_0:=\{x\in\mathbb{R}^2:0<x_1<\Lambda,\:x_2=h_1\}$.

For the incident field  $u^i=e^{{\rm i}(\alpha x_1-\beta x_2)}$, the DtN operator \eqref{a0} yields
\begin{align*}
Tu^i=\sum_{n\in\mathbb{Z}} {\rm i}\beta_nu_n^i(h_1)e^{\mathrm {i} \alpha_n x_1}=\mathrm {i}\beta_0e^{\mathrm {i}(\alpha x_1-\beta h_1)}
=\mathrm {i}\beta e^{\mathrm {i}(\alpha x_1-\beta h_1)}.
\end{align*}
For the total field, direct calculation gives the boundary condition
\begin{align*}
\partial_{x_2}u_\sigma=Tu_\sigma+g\quad \text{on}\ \Gamma_0,
\end{align*}
where
\begin{align*}
g=-2{\rm i} \beta e^{{\rm i}(\alpha x_1-\beta h_1)}.
\end{align*}
Consequently, the reduced boundary value problem can be considered in a bounded domain
\begin{align}\label{c1}
\begin{cases}
\nabla\cdot\left(\varepsilon_\sigma^{-1}\nabla u_\sigma\right)+\omega^2\mu u_\sigma+{\rm i}\sigma u_\sigma=0&\quad \text{in}\ \Omega,\\
\partial_{x_2}u_\sigma=Tu_\sigma+g&\quad \text{on}\ \Gamma_0,\\
\partial_{x_2}u_\sigma=0&\quad \text{on}\ \Gamma,
\end{cases}
\end{align}
where $\Omega:=\Omega_1\cup{\Omega_2}\cup S$, and
$
\Omega_1:=\{x\in\mathbb{R}^2:0<x_1<\Lambda,\:f(x_1)<x_2<h_1\},
$
as illustrated in   Figure \ref{fig2}.

\begin{figure}[h]
    \centering
    \includegraphics[width=0.45\textwidth]{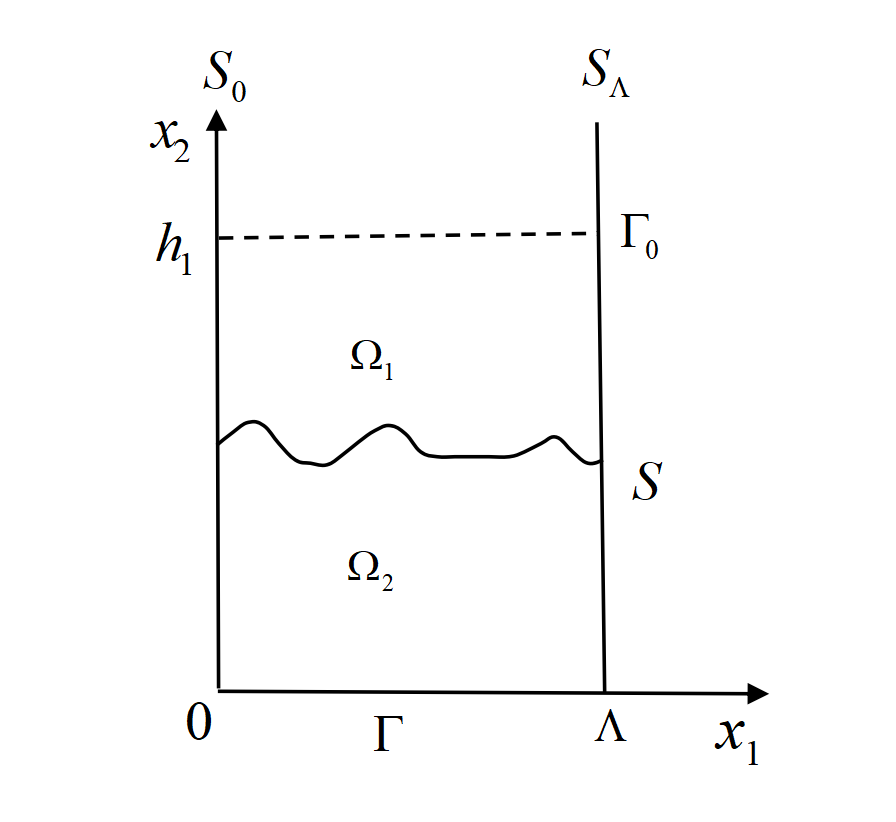}
    \caption{ Problem geometry with TBC}
    \label{fig2}
\end{figure}

Similarly, applying the transparent boundary condition to problem  \eqref{ab},
we derive the reduced boundary value problem:
\begin{align}\label{c2}
\begin{cases}
\nabla\cdot\left(\varepsilon_0^{-1}\nabla u_0\right)+\omega^2\mu u_0=0&\quad \text{in}\ \Omega,\\
\partial_{x_2}u_0=Tu_0+g&\quad \text{on}\ \Gamma_0,\\
\partial_{x_2}u_0=0&\quad \text{on}\ \Gamma.
\end{cases}
\end{align}

We begin by defining relevant function spaces. The quasi-periodic Lebesgue space is
\begin{align*}
L^2_{qp}(\Omega):=\{u\in L^2(\Omega): u\left(0,x_2\right)=u\left(\Lambda,x_2\right)e^{-{\rm i}\alpha\Lambda}\}.
\end{align*}
 The standard Sobolev space is given by
  \[H^1(\Omega):=\{u\in L^2(\Omega):\nabla u\in L^2(\Omega)^2\},\]
with its quasi-periodic subspace defined as
\begin{align*}
H_{qp}^1\left(\Omega\right):=\{u\in H^1\left(\Omega\right): u\left(0,x_2\right)=u\left(\Lambda,x_2\right)e^{- {\rm i} \alpha\Lambda}\}.
\end{align*}
For $s\in\mathbb{R}$, the quasi-periodic trace space  $H^s_{qp}(\Gamma)$ is equipped with the norm
\begin{align*}
\|u\|^2_{H^s(\Gamma)}=\Lambda\sum\limits_{n\in\mathbb{Z}}\left(1+\alpha^2_n\right)^s|u_n|^2,
\end{align*}
where $u_n$ are the Fourier coefficients of the quasi-periodic function $u(x)$. The dual space $H^{-s}_{qp}(\Gamma)$ is defined with respect to the  $L^2_{qp}(\Gamma)$- inner product:
\begin{align*}
\langle u,v\rangle_{\Gamma}:=\int_{\Gamma}u\overline {v}{\rm d}x_1=\Lambda\sum\limits_{n\in\mathbb{Z}}u_n\overline {v}_n,
\end{align*}
where $\overline {v}$ represents the conjugation of $v$.

Due to the sign inversion in the equation's coefficient, coercivity fails when applying variational methods to establish solution existence and uniqueness.  To address this, inspired by \cite{WOS:000278570900038}, we define an operator to recover coercivity. Define a bijective operator $\mathbb{T}$: $H^1_{qp}(\Omega)\to H^1_{qp}(\Omega)$ as
\begin{align*}
\mathbb{T}v=\begin{cases}
v_1\quad&\text{ in } \Omega_1,\\
-v_2+2\mathcal{R}(v|_S)\quad&\text{ in } \Omega_2,
\end{cases}
\end{align*}
where $v_1=v|_{\Omega_1}$, $v_2=v|_{\Omega_2}$, and $\mathcal{R}$: $H^{1/2}_{qp}(S)\to H^1_{qp}(\Omega_2)$ is a bounded linear extension operator satisfying
\begin{align*}
\mathcal{R}(\psi)|_S=\psi,\quad\forall \psi\in H^{1/2}_{qp}(S).
\end{align*}

Let us  present the main result of this paper.
\begin{theo}\label{TM}
Under the assumption that
\begin{align*}
\mathop{\operatorname{ sup}}\limits_{\sigma>0}\dfrac{|\varepsilon_2|+\sigma/\omega}{\varepsilon_2^2+(\sigma/\omega)^2}K<1,
\end{align*}
where $K=\dfrac{\langle\nabla\mathcal{R}(u_{\sigma}|_{S}),\nabla\mathcal{R}(u_{\sigma}|_{S})\rangle_{\Omega_2}}{\langle\varepsilon_1^{-1}\nabla u_{\sigma},\nabla u_{\sigma}\rangle_{\Omega_1}}$, the following hold:
\begin{enumerate}
\item For all \( \omega \in\mathbb{R}_+\setminus G \), where \( G \subset \mathbb{R}_+ \) is a discrete set and \( \mathbb{R}_+ := (0, \infty) \), the boundary value problem \eqref{c1} admits a unique solution \( u_\sigma \in H^1_{qp}(\Omega) \).

\item The boundary value problem \eqref{c2} has a unique solution \( u_0 \in H^1_{qp}(\Omega) \), which satisfies
\[
\| u_0 \|_{H^1(\Omega)} \leqslant C \| g \|_{H^{-1/2}(\Gamma_0)},
\]
where \( u_0 \) is the limit of solutions \( u_\sigma \) to \eqref{c1} as \( \sigma \to 0 \), and \( C > 0 \) is a constant independent of \( g \) and \( \sigma \).
\end{enumerate}

\end{theo}

\section{the complementing boundary condition}\label{cbc}
\subsection{Definition of the complementing boundary condition}
To prove Theorem \ref{TM}, we adopt the complementing boundary condition framework established by Agmon, Douglis, and Nirenberg  \cite{https://doi.org/10.1002/cpa.3160170104}. We first introduce key definitions and notation.

Consider the following boundary value problem
\begin{align}\label{def1}
\begin{cases}
\sum\limits_{j=1}^NL_{ij}(x;\partial)u_j(x)=F_i(x)\quad\text{in}\ \mathbb{R}^{n+1}_+,\\
\sum\limits_{j=1}^NB_{ij}(x;\partial)u_j(x)=\phi_i(x)\quad\text{on}\ \mathbb{R}^{n+1}_0,
\end{cases}
\end{align}
where the domains are defined as
\begin{align*}
\mathbb{R}^{n+1}_+:=\{x\in\mathbb{R}^{n+1}:\:x_{n+1}>0\},\quad\mathbb{R}^{n+1}_0:=\{x\in\mathbb{R}^{n+1}:\:x_{n+1}=0\},\quad i=1, \cdots, N,
\end{align*}
with the differential operator  $\partial=\left(\partial/{\partial x_1},\partial/{\partial x_2},\cdots,\partial/{\partial x_{n+1}}\right)$.
Let  $\triangle(x;\xi)={\rm det} (L'_{ij}(x;\xi))$, where $(L_{ij}^{'}(x;\xi))$ denotes  the  matrix consisting of the highest-order terms of $L_{ij}(x;\xi)$, and $\xi$ is a non-zero real vector.  The system \eqref{def1} is  elliptic in the sense of Douglis-Nirenberg if and only if
\begin{align*}
\triangle(x;\xi)\neq 0,\quad\forall\xi\in\mathbb{R}^{n+1}\setminus\{0\}.
\end{align*}
For $n>1$, Lopatinskii \cite{MR73828} proved that the order of $\triangle(x;\xi)$ is even (with respect to $\xi$). Decomposing the variables into $\xi'=(\xi_1,\xi_2,\cdots,\xi_n)$ and $\tau=\xi_{n+1}$,  the  determinant $\triangle(x;\xi)$ becomes a polynomial in $\tau$.
Crucially,  for any fixed nonzero real vector $\xi'$, if $\tau$ is a root of $\triangle(x;\xi',\tau)=0$, then $-\tau$ is a root of $\triangle(x;-\xi',-\tau)=0$. This symmetry implies that roots with positive/negative imaginary parts occur in conjugate pairs.

In this paper, we  analyze  a second-order $2\times2$ elliptic system in $\mathbb{R}^2$, where  for every $(x;\xi)$, $L_{ij}$ are second-order polynomials in $\xi$. Consequently,  the determinant $\triangle(x;\xi)$ becomes  a fourth-order polynomial in $\xi$, and the equation  $\triangle(x;\xi',\tau)=0$ possesses two pairs of complex conjugate  roots. Let $\tau_1^+(x;\xi')$ and $\tau_2^+(x;\xi')$ denote the roots with positive imaginary parts of the characteristic equation
$\triangle(x;\xi+\tau\vec{n})=0$, where $\tau$ is a complex variable, and  $\vec{n}$ and $\xi$ represent the normal and tangent vectors on $\mathbb{R}^2_0$, respectively.
The system is uniformly elliptic in the sense that there exist positive  constants $c_1>0, c_2>0 $,  satisfying
\begin{align*}
c_1|\xi|^2\leq|\triangle(x;\xi)|\leq c_2|\xi|^2,
\end{align*}
for every real vector $\xi=(\xi_1,\xi_2)$ and for every point $x$ in the closure of the considered domain.
Let $(L^*_{jk}(x,\xi+\tau\vec{n}))$ represent the adjoint matrix of $(L^{'}_{ij}(x;\xi+\tau\vec{n}))$, and $(B^{'}_{ij}(x;\xi))$ is defined as the matrix composed of the highest-order terms of $B_{ij}(x;\xi)$. Denote by
\begin{align*}
M^+(x;\xi',\tau)=\prod_{k=1}^2\left(\tau-\tau_k^+(x;\xi') \right).
\end{align*}
For the boundary problem \eqref{def1}, if the following definition is satisfied, then we refer to the conditions on the boundary $\mathbb{R}^2_0$ as the  complementing boundary conditions.

\begin{defi}\label{da}
( \cite[P42]{https://doi.org/10.1002/cpa.3160170104}) For any $x_0\in\mathbb{R}^2_0$ and non-zero real vector $\xi$ tangent to $\mathbb{R}^2_0$ at $x_0$, construct the  polynomial matrix
\begin{align*}
(B^{'}_{ij}(x_0;\xi+\tau\vec{n}))(L^*_{jk}(x_0;\xi+\tau\vec{n})),
\end{align*}
the complementing boundary condition requires that, under the modulus $M^+(x_0;\xi',\tau)$, the row vectors of this  matrix are  linearly independent. That is, if there exist constants $c_i, i=1, 2,$ such that
\begin{align*}
\sum_{i=1}^2c_i\sum_{j=1}^2(B^{'}_{ij})(L^*_{jk})\equiv 0\ ({\rm  mod}\:M^+),
\end{align*}
then  $c_1=c_2 =0$.
\end{defi}

The essence of the complementing boundary condition lies in algebraically coupling the boundary conditions with the intrinsic structure of the differential operator to ensure the well-posedness of the boundary value problem. Specifically, for the matrix $(L'_{ij})$ of an elliptic system, the adjoint matrix $(L^*_{jk})$ satisfies the identity $(L'_{ij})(L^*_{jk})=\triangle I$, thereby linking the solution space of the system to its determinant $\triangle$. For the  boundary problems,  roots of the discriminant polynomial $\triangle$ with positive imaginary parts characterize the solution decay
  in the normal direction.
  To capture the interaction between boundary operators and the differential structure, a polynomial matrix is constructed by combining $(B'_{ij})$ with $L^*_{jk}$. This coupling effectively reflects the interplay between the boundary conditions and the system's discriminant polynomial.
   Since we need to construct a polynomial matrix that captures the influence of the boundary conditions, the adjoint matrix is used to ``extract" non-trivial solutions at the zeros of the characteristic polynomial $\triangle$ ( i.e.,  $\triangle=0$) and to achieve ``matching" conditions when relating the boundary conditions to the solutions. For the well-posedness of the elliptic system, only the physically permissible decaying solutions should be retained. The modulo operation with respect to $M^+$  serves to filter out the decaying solutions while preserving the constraints imposed by the boundary conditions. If the row vectors remain linearly independent after the modulo $M^+$ operation, it indicates that the boundary condition ``completes" the deficiency of the differential operator, eliminating non-zero solutions. This is also a key prerequisite for establishing a priori estimates in an elliptic problem.
\subsection{Application of the complementing boundary condition and the $H^2$ regularity}\label{3.2}
Since not all equations defined on $\Omega$ are elliptic and  the function $f$ is not constant ( implying that the surface $S$ is not flat), it is necessary according to Definition \ref{da} to introduce variable transformations that  flatten $S$ and reformulate the original equations as the elliptic equations defined on the half-plane $\{x\in\mathbb{R}^2:\:x_2\geqslant 0\}$.

For the following transmission problem
\begin{align}\label{z18}
\begin{cases}
\nabla\cdot\left(\left(\varepsilon_2+\dfrac{{\rm i}\sigma}{\omega}\right)^{-1}\nabla u_\sigma\right)+\omega^2\mu_2u_\sigma+{\rm i}\sigma u_\sigma=0&\quad\text{in}\ \Omega_2,\\
\nabla\cdot\left(\varepsilon_1^{-1}\nabla u_\sigma\right)+\omega^2\mu_1u_\sigma=0&\quad\text{in}\ \Omega_1,\\
u_\sigma|_{\Omega_1}=u_\sigma|_{\Omega_2},\quad\varepsilon_1^{-1}\nabla u_\sigma|_{\Omega_1}\cdot\nu=\left(\varepsilon_2+\dfrac{{\rm i}\sigma}{\omega}\right)^{-1}\nabla u_\sigma|_{\Omega_2}\cdot\nu&\quad\text{on}\ S,
\end{cases}
\end{align}
we introduce the following variable transformations:
\begin{align*}
\Psi_{+}:\qquad\tilde{x}_1=\frac{2}{\Lambda}x_1-1,\quad\tilde{x}_2=\dfrac{x_2-f}{h_1-f},\quad f<x_2<h_1,
\end{align*}
\begin{align*}
\Psi_{-}:\qquad\tilde{x}_1=\frac{2}{\Lambda}x_1-1,\quad\tilde{x}_2=\dfrac{x_2-f}{f},\quad 0<x_2<f,
\end{align*}
i.e., $\Psi_{+}$ and $\Psi_{-}$ are invertible transformations from $\Omega_1$ to $D_+$ and $\Omega_2$ to $D_-$, respectively, where
\begin{align*}
D_+:=\{\tilde{x}\in\mathbb{R}^2: -1<\tilde{x}_1<1,\:0<\tilde{x}_2<1\},
\end{align*}
\begin{align*}
D_-:=\{\tilde{x}\in\mathbb{R}^2: -1<\tilde{x}_1<1,\:-1<\tilde{x}_2<0\}.
\end{align*}

Applying  the partial derivatives
\begin{align*}
\frac{\partial\tilde{x}_1}{\partial x_1}=\frac{2}{\Lambda},\quad\frac{\partial\tilde{x}_1}{\partial x_2}=0,\quad\frac{\partial\tilde{x}_2}{\partial x_1}=f^{'}\frac{\tilde{x}_2-1}{h_1-f},\quad\frac{\partial\tilde{x}_2}{\partial x_2}=\frac{1}{h_1-f}, \quad f<x_2<h_1,
\end{align*}
\begin{align*}
\frac{\partial\tilde{x}_1}{\partial x_1}=\frac{2}{\Lambda},\quad\frac{\partial\tilde{x}_1}{\partial x_2}=0,\quad\frac{\partial\tilde{x}_2}{\partial x_1}=-f^{'}\frac{\tilde{x}_2+1}{f},\quad\frac{\partial\tilde{x}_2}{\partial x_2}=\frac{1}{f},\quad 0<x_2<f,
\end{align*}
we derive the transformed gradients:
\begin{align*}
\nabla_xu_\sigma&=\left(\dfrac{\partial u_\sigma}{\partial x_1},\dfrac{\partial u_\sigma}{\partial x_2}\right)^T=\left( {\begin{array}{*{20}{c}}
\dfrac{\partial\tilde{x}_1}{\partial x_1}&\dfrac{\partial\tilde{x}_2}{\partial x_1}\\
\dfrac{\partial\tilde{x}_1}{\partial x_2}&\dfrac{\partial\tilde{x}_2}{\partial x_2}
\end{array}} \right)\left( {\begin{array}{*{20}{c}}
\dfrac{\partial u_\sigma}{\partial\tilde{x}_1}\\ \dfrac{\partial u_\sigma}{\partial\tilde{x}_2}
\end{array}} \right)\\&=\left( {\begin{array}{*{20}{c}}
\dfrac{2}{\Lambda}&\dfrac{f^{'}(\tilde{x}_2-1)}{h_1-f}\\
0&\dfrac{1}{h_1-f}
\end{array}} \right)\left( {\begin{array}{*{20}{c}}
\dfrac{\partial u_\sigma}{\partial\tilde{x}_1}\\ \dfrac{\partial u_\sigma}{\partial\tilde{x}_2}
\end{array}} \right)\triangleq\left(\dfrac{\partial\Psi_+}{\partial x}\right)^T\nabla_{\tilde{x}}u_\sigma,\quad f<x_2<h_1
\end{align*}
and
\begin{align*}
\nabla_xu_\sigma&=\left(\dfrac{\partial u_\sigma}{\partial x_1},\dfrac{\partial u_\sigma}{\partial x_2}\right)^T=\left( {\begin{array}{*{20}{c}}
\dfrac{\partial\tilde{x}_1}{\partial x_1}&\dfrac{\partial\tilde{x}_2}{\partial x_1}\\
\dfrac{\partial\tilde{x}_1}{\partial x_2}&\dfrac{\partial\tilde{x}_2}{\partial x_2}
\end{array}} \right)\left( {\begin{array}{*{20}{c}}
\dfrac{\partial u_\sigma}{\partial\tilde{x}_1}\\ \dfrac{\partial u_\sigma}{\partial\tilde{x}_2}
\end{array}} \right)\\&=\left( {\begin{array}{*{20}{c}}
\dfrac{2}{\Lambda}&\dfrac{-f^{'}(\tilde{x}_2+1)}{f}\\
0&\dfrac{1}{f}
\end{array}} \right)\left( {\begin{array}{*{20}{c}}
\dfrac{\partial u_\sigma}{\partial\tilde{x}_1}\\ \dfrac{\partial u_\sigma}{\partial\tilde{x}_2}
\end{array}} \right)\triangleq\left(\dfrac{\partial\Psi_-}{\partial x}\right)^T\nabla_{\tilde{x}}u_\sigma,\quad 0<x_2<f.
\end{align*}
Therefore, for any $\psi_1\in H^1_{qp}(\Omega_1)$ and $\psi_2\in H^1_{qp}(\Omega_2)$, we obtain
\begin{align*}
\int_{\Omega_1}&(\nabla\cdot(\varepsilon_1^{-1}\nabla u_\sigma(x))+\omega^2\mu_1u_\sigma(x))\psi_1(x){\rm d}x\\
=&\int_{D_+}\left(\nabla_{\tilde{x}}\cdot\left(\varepsilon_1^{-1}\left(\dfrac{\partial\Psi_+}{\partial x}\right)(\Psi^{-1}_+(\tilde{x}))\left(\dfrac{\partial\Psi_+}{\partial x}\right)^T(\Psi^{-1}_+(\tilde{x}))\nabla_{\tilde{x}}u_\sigma(\Psi_+^{-1}(\tilde{x}))\right)\right.\\
&\left.+\omega^2\mu_1u_\sigma(\Psi_+^{-1}(\tilde{x}))\right)\psi_1(\Psi_+^{-1}(\tilde{x}))J_+{\rm d}\tilde{x}
\end{align*}
and
\begin{align*}
\int_{\Omega_2}&\left(\nabla\cdot\left(\left(\varepsilon_2+\dfrac{{\rm i}\sigma}{\omega}\right)^{-1}\nabla u_\sigma(x)\right)+\omega^2\mu_2u_\sigma(x)+{\rm i}\sigma u_\sigma(x)\right)\psi_2(x){\rm d}x\\
=&\int_{D_-}\left(\nabla_{\tilde{x}}\cdot\left(\left(\varepsilon_2+\dfrac{\rm i\sigma}{\omega}\right)^{-1}\left(\dfrac{\partial\Psi_-}{\partial x}\right)(\Psi^{-1}_-(\tilde{x}))\left(\dfrac{\partial\Psi_-}{\partial x}\right)^T(\Psi^{-1}_-(\tilde{x}))\nabla_{\tilde{x}}u_\sigma(\Psi_-^{-1}(\tilde{x}))\right)\right.\\
&\left.+\omega^2\mu_2u_\sigma(\Psi_-^{-1}(\tilde{x}))+{\rm i}\sigma u_\sigma(\Psi_-^{-1}(\tilde{x}))\right)\psi_2(\Psi_-^{-1}(\tilde{x}))J_-{\rm d}\tilde{x},
\end{align*}
where $J_{\pm}=|{\rm det} \partial\Psi_{\pm}^{-1}(\tilde{x})|, \tilde{x}\in D_\pm$, i.e.,
\begin{align*}
J_+=\begin{vmatrix}
\dfrac{\partial x_1}{\partial\tilde{x}_1}&\dfrac{\partial x_1}{\partial\tilde{x}_2}\\
\dfrac{\partial x_2}{\partial\tilde{x}_1}&\dfrac{\partial x_2}{\partial\tilde{x}_2}
\end{vmatrix}
=\begin{vmatrix}
\dfrac{\Lambda}{2}&\dfrac{h_1-f}{f^{'}(\tilde{x}_2-1)}\\
0&h_1-f
\end{vmatrix}
=\dfrac{\Lambda}{2}(h_1-f),
\end{align*}
\begin{align*}
J_-=\begin{vmatrix}
\dfrac{\partial x_1}{\partial\tilde{x}_1}&\dfrac{\partial x_1}{\partial\tilde{x}_2}\\
\dfrac{\partial x_2}{\partial\tilde{x}_1}&\dfrac{\partial x_2}{\partial\tilde{x}_2}
\end{vmatrix}
=\begin{vmatrix}
\dfrac{\Lambda}{2}&\dfrac{f}{-f^{'}(\tilde{x}_2+1)}\\
0&f
\end{vmatrix}
=\dfrac{\Lambda}{2}f.
\end{align*}
Hence, under these transformations, the system becomes
\begin{align}\label{q1}
\begin{cases}
\left(\nabla\cdot\left(\left(\varepsilon_2+\dfrac{{\rm i}\sigma}{\omega}\right)^{-1}\Psi^*_-\nabla u_\sigma\right)+\omega^2\mu_2u_\sigma+{\rm i}\sigma u_\sigma\right)J_-=0&\ \text{in}\ D_-,\\
\left(\nabla\cdot\left(\varepsilon_1^{-1}\Psi^*_+\nabla u_\sigma\right)+\omega^2\mu_1u_\sigma\right)J_+=0&\ \text{in}\ D_+,\\
u_\sigma|_{D_+}=u_\sigma|_{D_-}&\ \text{on}\ \partial D_+\cap\{\tilde{x}\in\mathbb{R}^2: \tilde{x}_2=0\},\\
\varepsilon_1^{-1}J_+\left(\Psi^*_+\nabla u_\sigma|_{D_+}\right)^T\cdot\vec{n}=\left(\varepsilon_2+\dfrac{{\rm i}\sigma}{\omega}\right)^{-1}J_-\left(\Psi^*_-\nabla u_\sigma|_{D_-}\right)^T\cdot\vec{n}&\ \text{on}\ \partial D_+\cap\{\tilde{x}\in\mathbb{R}^2: \tilde{x}_2=0\},
\end{cases}
\end{align}
where $\vec{n}=(0,1)^T$  and
\begin{align*}
\Psi^*_\pm=\left(\dfrac{\partial\Psi_\pm}{\partial x}\right)(\Psi_\pm^{-1}(\tilde{x}))\left(\dfrac{\partial\Psi_\pm}{\partial x}\right)^T(\Psi_\pm^{-1}(\tilde{x})),\quad \tilde{x}\in D_\pm.
\end{align*}

We set $u_\sigma|_{D_-}=\left(\varepsilon_2+\dfrac{{\rm i}\sigma}{\omega}\right)\tilde{u}_\sigma$ and substitute $\left(\varepsilon_2+\dfrac{{\rm i}\sigma}{\omega}\right)\tilde{u}_\sigma$ into \eqref{q1}. It is clear that $|u_\sigma|=C|\tilde{u}_\sigma|$, where $C$ is a constant. To simplify the notation, we will continue to denote $\tilde{u}_\sigma$ as $u_\sigma$. Thus, we obtain the following elliptic system
\begin{align}\label{q2}
\begin{cases}
\left(\nabla\cdot\left(\Psi^*_-\nabla u_\sigma\right)+\left(\varepsilon_2+\dfrac{{\rm i}\sigma}{\omega}\right)\left(\omega^2\mu_2u_\sigma+{\rm i}\sigma u_\sigma\right)\right)J_-=0&\quad\text{in}\ D_-,\\
\left(\nabla\cdot\left(\varepsilon_1^{-1}\Psi^*_+\nabla u_\sigma\right)+\omega^2\mu_1u_\sigma\right)J_+=0&\quad\text{in}\ D_+,\\
u_\sigma|_{D_+}=\left(\varepsilon_2+\dfrac{{\rm i}\sigma}{\omega}\right)u_\sigma|_{D_-}&\quad\text{on}\ \partial D_+\cap\{\tilde{x}\in\mathbb{R}^2: \tilde{x}_2=0\},\\
\varepsilon_1^{-1}J_+\left(\Psi^*_+\nabla u_\sigma|_{D_+}\right)^T\cdot\vec{n}=J_-\left(\Psi^*_-\nabla u_\sigma|_{D_-}\right)^T\cdot\vec{n}&\quad\text{on}\ \partial D_+\cap\{\tilde{x}\in\mathbb{R}^2: \tilde{x}_2=0\}.
\end{cases}
\end{align}
Since the  equations in \eqref{q2} are defined respectively  in $D_+$ and $D_-$, we need to introduce a variable transformation that enables the equation originally defined in $D_-$ to be redefined in $D_+$ after the transformation.

For $\tilde{x}=(\tilde{x}_1,\tilde{x}_2)\in D_+$, define the transformed function
\begin{align*}
w_\sigma(\tilde{x}_1,\tilde{x}_2)=u_\sigma(\tilde{x}_1,-\tilde{x}_2),
\end{align*}
which satisfies
\begin{align*}
\left(\nabla\cdot\left(\hat{\Psi}_+^*\nabla w_\sigma\right)+\left(\varepsilon_2+\dfrac{{\rm i}\sigma}{\omega}\right)(\omega^2\mu_2w_\sigma+{\rm i}\sigma w_\sigma)\right)J_-=0\quad \text{in}\ D_+,
\end{align*}
\begin{align*}
w_\sigma=u_\sigma,\quad-J_-\left(\hat{\Psi}_+^*\nabla w_\sigma\right)^T\cdot\vec{n}=J_-\left(\Psi_-^*\nabla u_\sigma\right)^T\cdot\vec{n}\quad\text{on}\ \partial D_+\cap\{x\in\mathbb{R}^2,\:x_2=0\},
\end{align*}
where ${\Psi}_+^*$
is defined component-wise as
\begin{align*}
(\hat{\Psi}^*_+)_{ij}(\tilde{x}_1,\tilde{x}_2)=\begin{cases}
(\Psi_-^*)_{ij}(\tilde{x}_1,-\tilde{x}_2),&\quad i=j=1,2,\\
-(\Psi_-^*)_{ij}(\tilde{x}_1,-\tilde{x}_2),&\quad i=1,\ j=2\ \text{or}\ i=2,\ j=1.
\end{cases}
\end{align*}
After  the variable transformation, the system \eqref{q2} becomes
\begin{align}\label{z4}
\begin{cases}
\left(\nabla\cdot\left(\hat{\Psi}^*_+\nabla w_\sigma\right)+\left(\varepsilon_2+\dfrac{{\rm i}\sigma}{\omega}\right)\left(\omega^2\mu_2w_\sigma+{\rm i}\sigma w_\sigma\right)\right)J_-=0&\quad\text{in}\ D_+,\\
\left(\nabla\cdot\left(\varepsilon_1^{-1}\Psi^*_+\nabla u_\sigma\right)+\omega^2\mu_1u_\sigma\right)J_+=0&\quad\text{in}\ D_+,\\
u_\sigma=\left(\varepsilon_2+\dfrac{{\rm i}\sigma}{\omega}\right)w_\sigma&\quad\text{on}\ \partial D_+\cap\{\tilde{x}\in\mathbb{R}^2: \tilde{x}_2=0\},\\
\varepsilon_1^{-1}J_+\left(\Psi^*_+\nabla u_\sigma\right)^T\cdot\vec{n}=-J_-\left(\hat{\Psi}^*_+\nabla w_\sigma\right)^T\cdot\vec{n}&\quad\text{on}\ \partial D_+\cap\{\tilde{x}\in\mathbb{R}^2: \tilde{x}_2=0\}.
\end{cases}
\end{align}
Therefore,  via variable transformations, the two equations in \eqref{z18} are reformulated as elliptic equations in the half-plane $\{\tilde{x}\in\mathbb{R}^2:\:\tilde{x}_2\geqslant 0\}$
 with a partially flat boundary. The subsequent lemma aims to verify that the boundary conditions defined on $\partial D_+\cap\{\tilde{x}\in\mathbb{R}^2: \tilde{x}_2=0\}$ in \eqref{z4} satisfy the  complementing boundary conditions. For brevity, we denote $\tilde{x}$ as $x$ hereafter.

\begin{lemm}\label{az}
The boundary conditions on $\partial D_+\cap\{x\in\mathbb{R}^2: x_2=0\}$ in \eqref{z4} are complementing boundary conditions.
\end{lemm}
\begin{proof}
According to Definition \ref{da}, the proof proceeds in two steps.

\textbf {Step 1.} We prove that  equation $\triangle (x; \xi) =0$ possesses  two roots with positive imaginary parts when considered as an equation in $\tau$.

By \cite[Theorem 9.1]{MR1770682}, let $\xi=(\xi_1,\xi_2)$, where $\xi_1\in\mathbb{R}\setminus\{0\}, \xi_2=\tau$, and $\tau$ can take all complex numbers.
Define the matrix
\begin{align*}
(L_{ij}^{'}(x;\xi))=& \left( {\begin{array}{*{20}{c}}
\varepsilon_1^{-1}J_+(x)(\Psi^*_+)_{ij}(x)\xi_i\xi_j&0\\
0&J_-(x)(\hat{\Psi}^*_+)_{ij}(x)\xi_i\xi_j
\end{array}} \right),\quad i,j=1,2.
\end{align*}
Let $\triangle(x;\xi)={\rm det}(L_{ij}^{'}(x;\xi))$, which implies   that  $\triangle(x;\xi)$ is a fourth-order polynomial in $\xi_i$. For any two linearly independent vectors $\gamma_1=(\gamma_{1,1},\gamma_{1,2})$ and $\gamma_2=(\gamma_{2,1},\gamma_{2,2})$, we have
\begin{align*}
\triangle&(x;\gamma_1+\tau\gamma_2)=\\ &\varepsilon_1^{-1}J_+(x)\left((\Psi^*_+)_{ij}(x)(\gamma_1+\tau\gamma_2)_i(\gamma_1+\tau\gamma_2)_j\right)J_-(x)\left((\hat{\Psi}^*_+)_{ij}(x)(\gamma_1+\tau\gamma_2)_i(\gamma_1+\tau\gamma_2)_j\right).
\end{align*}
We verify that $\triangle(x;\gamma_1+\tau\gamma_2)=0$ has two roots with positive imaginary parts when considered as an equation in $\tau$. That is, we need to confirm  that the following two equations  contain a pair of complex conjugate roots
\begin{align}\label{w1}
J_+(x)\left((\Psi^*_+)_{ij}(x)(\gamma_1+\tau\gamma_2)_i(\gamma_1+\tau\gamma_2)_j\right)=0
\end{align}
and
\begin{align}\label{w2}
J_-(x)\left((\hat{\Psi}^*_+)_{ij}(x)(\gamma_1+\tau\gamma_2)_i(\gamma_1+\tau\gamma_2)_j\right)=0.
\end{align}
For \eqref{w1}, we compute
\begin{align*}
&\left(\dfrac{4}{\Lambda}(h_1-f)\gamma_{1,1}\gamma_{2,1}+2f^{'}(x_2-1)(\gamma_{1,1}\gamma_{2,2}+\gamma_{1,2}\gamma_{2,1})+\Lambda\dfrac{f^{'2}(x_2-1)^2+1}{h_1-f}\gamma_{1,2}\gamma_{2,2}\right)^2\\
&-4\left(\dfrac{2}{\Lambda}(h_1-f)\gamma_{2,1}^2+2f^{'}(x_2-1)\gamma_{2,1}\gamma_{2,2}+\dfrac{\Lambda}{2}\dfrac{f^{'2}(x_2-1)^2+1}{h_1-f}\gamma_{2,2}^2\right)\left(\dfrac{2}{\Lambda}(h_1-f)\gamma_{1,1}^2\right.\\
&\left.+2f^{'}(x_2-1)\gamma_{1,1}\gamma_{1,2}+\dfrac{\Lambda}{2}\dfrac{f^{'2}(x_2-1)^2+1}{h_1-f}\gamma_{1,2}^2\right)\\
=&-4\left(\gamma_{1,2}\gamma_{2,1}-\gamma_{1,1}\gamma_{2,2}\right)^2<0,
\end{align*}
which  implies that  \eqref{w1} has a pair of  complex conjugate roots. Similarly, \eqref{w2} also yields  a pair of complex conjugate  roots. Therefore, $\triangle(x;\gamma_1+\tau\gamma_2)=0$ indeed  has two roots with positive imaginary parts.

\textbf{ Step 2.} We prove that the rows of the matrix $(D_{ik})=(B'_{ij})(L^*_{jk})$ are linearly independent modulo $M^+$.

From Definition \ref{da}, the matrices are defined as
\begin{align*}
(L^*_{jk}(x;\xi))=& \left( {\begin{array}{*{20}{c}}
J_-(x)(\hat{\Psi}^*_+)_{ij}(x)\xi_i\xi_j&0\\
0&\varepsilon_1^{-1}J_+(x)(\Psi^*_+)_{ij}(x)\xi_i\xi_j
\end{array}} \right)
\end{align*}
and
\begin{align*}
(B'_{ij}(x;\xi))=& \left( {\begin{array}{*{20}{c}}
1&-\left(\varepsilon_2+\dfrac{{\rm i}\sigma}{\omega}\right)\\
\varepsilon_1^{-1}J_+(x)(\Psi^*_+)_{2i}(x)\xi_i&J_-(x)(\hat{\Psi}^*_+)_{2i}(x)\xi_i
\end{array}} \right).
\end{align*}
Consequently,
\begin{align*}
&(D_{ik}(x;\xi))=\varepsilon_1^{-1}\cdot\\& \left( {\begin{array}{*{20}{c}}
\varepsilon_1J_-(x)(\hat{\Psi}^*_+)_{ij}(x)\xi_i\xi_j&-\left(\varepsilon_2+\dfrac{{\rm i}\sigma}{\omega}\right)J_+(x)\left((\Psi^*_+)_{ij}(x)\xi_i\xi_j\right)\\
J_+(x)\left((\Psi^*_+)_{2i}(x)\xi_i\right)J_-(x)\left((\hat{\Psi}^*_+)_{ij}(x)\xi_i\xi_j\right)&J_-(x)\left((\hat{\Psi}^*_+)_{2i}(x)\xi_i\right)J_+(x)\left((\Psi^*_+)_{ij}(x)\xi_i\xi_j\right)
\end{array}} \right).
\end{align*}
Let
\begin{align*}
M^+(x;\xi_1,\tau)=(\tau-\tau_1^+(x,\xi_1))(\tau-\tau_2^+(x,\xi_1)),
\end{align*}
according to Definition \ref{da}, we aim to  verify that the row vectors of $(D_{ik}(x;\xi))({\rm mod}\: M^+)$ are linearly independent  on $\partial D_+\cap\{x\in\mathbb{R}^2: x_2=0\}$, i.e.,
\begin{align}\label{z17}
\sum_{i=1}^2c_iD_{ik}(x;\xi)\equiv 0\ (mod\: M^+)
\end{align}
iff $c_1=c_2=0$, where $x\in\partial D_+\cap\{x\in\mathbb{R}^2: x_2=0\}$, $k=1,2$.

Treating   $J_+(x)(\Psi^*_+)_{ij}(x)\xi_i\xi_j$ as a polynomial in $\tau$, we factorize
\begin{align}\label{z19}
J_+(x)(\Psi^*_+)_{ij}(x)\xi_i\xi_j=J_+(x)(\Psi^*_+)_{22}(x)(\tau-\tau_1^+(x,\xi_1))(\tau-\tau_1^-(x,\xi_1)),
\end{align}
where $\tau_1^\pm (x;\xi_1)$ denotes the root with a positive/ negative  imaginary part  of $J_+(x)(\Psi^*_+)_{ij}(x)\xi_i\xi_j=0$. Similarly,
\begin{align}\label{z20}
J_-(x)(\hat{\Psi}^*_+)_{ij}(x)\xi_i\xi_j=J_-(x)(\hat{\Psi}^*_+)_{22}(x)(\tau-\tau_2^+(x,\xi_1))(\tau-\tau_2^-(x,\xi_1)),
\end{align}
where $\tau_2^\pm(x;\xi_1)$ represents the root with a positive/
negative  imaginary part of $J_-(x)(\hat{\Psi}^*_+)_{ij}(x)\xi_i\xi_j=0$. According to the definition of the modulus operation in modern algebra, \eqref{z17} implies that there exists $k_1, k_2\in\mathbb{C}$ such that
\begin{align*}
c_1J_-(\hat{\Psi}_+^*)_{ij}\xi_i\xi_j+c_2\varepsilon_1^{-1}J_+\left((\Psi_+^*)_{2i}\xi_i\right)J_-\left((\hat{\Psi}_+^*)_{ij}\xi_i\xi_j\right)=k_1M^+,
\end{align*}
\begin{align*}
-c_1\left(\varepsilon_2+\frac{{\rm i}\sigma}{\omega}\right)\varepsilon_1^{-1}J_+\left((\Psi_+^*)_{ij}\xi_i\xi_j\right)+c_2\varepsilon_1^{-1}J_-\left((\hat{\Psi}_+^*)_{2i}\xi_i\right)J_+\left((\Psi_+^*)_{ij}\xi_i\xi_j\right)=k_2M^+.
\end{align*}
Based on \eqref{z19} and \eqref{z20}, we have
\begin{align*}
c_1J_-&(\hat{\Psi}^*_+)_{22}(\tau-\tau_2^+(\xi_1))(\tau-\tau_2^-(\xi_1))\\
&+c_2\varepsilon_1^{-1}J_+\left((\Psi_+^*)_{2i}\xi_i\right)J_-(\hat{\Psi}^*_+)_{22}(\tau-\tau_2^+(\xi_1))(\tau-\tau_2^-(\xi_1))=k_1M^+,
\end{align*}
\begin{align*}
-c_1&\left(\varepsilon_2+\frac{{\rm i}\sigma}{\omega}\right)\varepsilon_1^{-1}J_+(\Psi^*_+)_{22}(\tau-\tau_1^+(\xi_1))(\tau-\tau_1^-(\xi_1))\\
&+c_2\varepsilon_1^{-1}J_-\left((\hat{\Psi}_+^*)_{2i}\xi_i\right)J_+(\Psi^*_+)_{22}(\tau-\tau_1^+(\xi_1))(\tau-\tau_1^-(\xi_1))=k_2M^+.
\end{align*}
Comparing the coefficients of $\tau$ on both sides of the two equations, it can be seen that there exists $\lambda_1,\lambda_2\in\mathbb{C}$ such that
\begin{align*}
\left(c_1+c_2\varepsilon_1^{-1}J_+((\Psi_+^*)_{2i}\xi_i)\right)&J_-(\hat{\Psi}^*_+)_{22}(\tau-\tau_2^+(\xi_1))(\tau-\tau_2^-(\xi_1))\\
&=c_2\varepsilon_1^{-1}J_+(\Psi^*_+)_{22}J_-(\hat{\Psi}^*_+)_{22}(\tau-\lambda_1)M^+,
\end{align*}
\begin{align*}
\left(-c_1\left(\varepsilon_2+\frac{{\rm i}\sigma}{\omega}\right)\varepsilon_1^{-1}+c_2\varepsilon_1^{-1}J_-\left((\hat{\Psi}_+^*)_{2i}\xi_i\right)\right)&J_+(\Psi_+^*)_{22}(\tau-\tau_1^+(\xi_1))(\tau-\tau_1^-(\xi_1))\\
&=c_2\varepsilon_1^{-1}J_+(\Psi^*_+)_{22}J_-(\hat{\Psi}^*_+)_{22}(\tau-\lambda_2)M^+.
\end{align*}
Considering these two equations as a system of equations in terms of $c_1$ and $c_2$, we will prove by contradiction that $c_1=c_2=0$.
If $c_1$ and $c_2$ are not both zero, then the determinant of the coefficient matrix (considered as a polynomial in $\tau$) is identically zero, i.e.,
\begin{align*}
&(\tau-\tau_1^-)(\tau-\tau_2^-)\left(J_-\left((\hat{\Psi}_+^*)_{2i}\xi_i\right)+\left(\varepsilon_2+\dfrac{{\rm i}\sigma}{\omega}\right)\varepsilon_1^{-1}J_+\left((\Psi_+^*)_{2i}\xi_i\right)\right)\\
\equiv&\dfrac{f^{'2}+1}{f}\dfrac{\Lambda}{2}(\tau-\tau_2^-)(\tau-\lambda_2)(\tau-\tau_2^+)+\left(\varepsilon_2+\dfrac{{\rm i}\sigma}{\omega}\right)\varepsilon_1^{-1}\dfrac{f^{'2}+1}{h_1-f}\dfrac{\Lambda}{2}(\tau-\tau_1^-)(\tau-\lambda_1)(\tau-\tau_1^+).
\end{align*}
Since the left side of the equation is the product of two polynomials in $\tau$, we have $\lambda_1=\tau_2^-$ and $\lambda_2=\tau_1^-$. Then
\begin{align*}
J_-\left((\hat{\Psi}_+^*)_{2i}\xi_i\right)+&\left(\varepsilon_2+\dfrac{{\rm i}\sigma}{\omega}\right)\varepsilon_1^{-1}J_+\left((\Psi_+^*)_{2i}\xi_i\right)\\
&\equiv\dfrac{f^{'2}+1}{f}\dfrac{\Lambda}{2}(\tau-\tau_2^+)
+\left(\varepsilon_2+\dfrac{{\rm i}\sigma}{\omega}\right)\varepsilon_1^{-1}\dfrac{f^{'2}+1}{h_1-f}\dfrac{\Lambda}{2}(\tau-\tau_1^+).
\end{align*}
By simple calculation, we can obtain
\begin{align}\label{z10}
f^{'}\xi_1-\left(\varepsilon_2+\dfrac{{\rm i}\sigma}{\omega}\right)\varepsilon_1^{-1}f^{'}\xi_1
+\dfrac{f^{'2}+1}{f}\dfrac{\Lambda}{2}\tau_2^++\left(\varepsilon_2+\dfrac{{\rm i}\sigma}{\omega}\right)\varepsilon_1^{-1}\dfrac{f^{'2}+1}{h_1-f}\dfrac{\Lambda}{2}\tau_1^+=0.
\end{align}

Since $\tau_1^+$ is the root with positive  imaginary part of $J_+(\Psi_+^*)_{ij}(x)\xi_i\xi_j=0$ and $\tau_2^+$ is the root with positive imaginary part of $J_-(\hat{\Psi}_+^*)_{ij}(x)\xi_i\xi_j=0$, we derive
\begin{align*}
\tau_1^\pm=\dfrac{h_1-f}{\Lambda(f^{'2}+1)}(2f^{'}\xi_1\pm2{\rm i}\xi_1),
\quad\tau_2^\pm=\dfrac{f}{\Lambda(f^{'2}+1)}(-2f^{'}\xi_1\pm 2 {\rm i}\xi_1).
\end{align*}
Consequently, their real parts
\begin{align*}
\operatorname{Re} \tau_1^+=\dfrac{h_1-f}{\Lambda(f^{'2}+1)}2f^{'}\xi_1,\quad \operatorname{Re}\tau_2^+=-\dfrac{f}{\Lambda(f^{'2}+1)}2f^{'}\xi_1.
\end{align*}
According to \eqref{z10}, by considering the real and imaginary parts separately, we obtain
\begin{align*}
\begin{cases}
f^{'}\xi_1-\varepsilon_2\varepsilon_1^{-1}f^{'}\xi_1+
\dfrac{f^{'2}+1}{f}\dfrac{\Lambda}{2}\operatorname{Re}\tau_2^+
+\varepsilon_2\varepsilon_1^{-1}\dfrac{f^{'2}+1}{h_1-f}\dfrac{\Lambda}{2}\operatorname{Re}\tau_1^+
-\dfrac{\sigma}{\omega}\varepsilon_1^{-1}\dfrac{f^{'2}+1}{h_1-f}\dfrac{\Lambda}{2}\operatorname{Im}\tau_1^+=0,\\
-\dfrac{\sigma}{\omega}\varepsilon_1^{-1}f^{'}\xi_1+\dfrac{f^{'2}+1}{f}\dfrac{\Lambda}{2}\operatorname{Im}\tau_2^+
+\dfrac{\sigma}{\omega}\varepsilon_1^{-1}\dfrac{f^{'2}+1}{h_1-f}\dfrac{\Lambda}{2}\operatorname{Re}\tau_1^++\varepsilon_2\varepsilon_1^{-1}\dfrac{f^{'2}+1}{h_1-f}\dfrac{\Lambda}{2} \operatorname{Im}\tau_1^+=0.
\end{cases}
\end{align*}
Substituting $\operatorname{Re}\tau_1^+$ and $\operatorname{Re}\tau_2^+$ into  the above  two equations, it reduces to
\begin{align*}
\begin{cases}
-\dfrac{\sigma}{\omega}\varepsilon_1^{-1}\dfrac{{f'}^2+1}{h_1-f}\dfrac{\Lambda}{2}\operatorname{Im}\tau_1^+\equiv 0,\\
\dfrac{{f'}^2+1}{f}\dfrac{\Lambda}{2}\operatorname{Im}\tau_2^++\varepsilon_2\varepsilon_1^{-1}\dfrac{{f'}^2+1}{h_1-f}\dfrac{\Lambda}{2}\operatorname{Im}\tau_1^+\equiv 0.
\end{cases}
\end{align*}
It follows from  $\operatorname{Im}\tau_1^+>0$ that  $-\dfrac{\sigma}{\omega}\varepsilon_1^{-1}\dfrac{{f'}^2+1}{h_1-f}\dfrac{\Lambda}{2}\operatorname{Im}\tau_1^+\neq 0$, forming a contradiction, which means $c_1=c_2=0$. Therefore, the boundary conditions on $\partial D_+\cap\{x\in\mathbb{R}^2: x_2=0\}$ are complementing boundary conditions.
\end{proof}

The following establishes the boundary  $H^2$-regularity result for the solutions of \eqref{z4}. The proof is based on the definition of $H^2(D_+)$.
\begin{lemm}\label{z5}
Let $u_\sigma, w_\sigma\in H^1(D_+)$ be  solutions to \eqref{z4}. Then for any $0<r<1$,
 \[u_\sigma, w_\sigma\in H^2(B_+(0,r)),\]
  where $B_+(0,r):=\{x\in\mathbb{R}^2_+:\:x_1^2+x_2^2<r^2\}$.
\end{lemm}
\begin{proof}
Since $u_\sigma\in H^1(D_+)$, according to the definition of the Sobolev space $H^2(D_+)$, it suffices to prove the boundedness of  $\|(u_\sigma)_{x_ix_j}\|_{L^2(B_+(0,r))}$ for all   $i,j=1,2$.
The proof proceeds in two stages.

\textbf {Step 1.} We prove the boundedness  of  $\|(u_\sigma)_{x_ix_j}\|_{L^2(B_+(0,r))}$  for all indices $i, j$, except  $(i, j)=(2, 2)$.

Define a smooth cutoff function  $\eta$ satisfying
\begin{align*}
\begin{cases}
\eta\equiv1&\quad\text{in}\ B(0,r),\\
\eta\equiv0&\quad\text{in}\ \mathbb{R}^2\setminus B(0,1),\\
0\leqslant\eta\leqslant1&\quad\text{in}\ B(0,1)\setminus B(0,r),
\end{cases}
\end{align*}
which induces the properties
\begin{align*}
\begin{cases}
\eta\equiv1&\quad\text{in}\ B(0,r)\cap D_+,\\
\eta=0&\quad\text{on}\ \partial D_+\setminus\left(\partial D_+\cap\{x\in\mathbb{R}^2:\:x_2=0\}\right),\\
\eta\neq 0&\quad\text{on}\ \partial D_+\cap\{x\in\mathbb{R}^2:\:x_2=0\}.
\end{cases}
\end{align*}
Multiplying both sides of  equation  \eqref{z4} by $\varphi\in H^1(D_+)$, integrating over $D_+$, and applying integration by parts with the boundary condition on $\partial D_+\cap\{x\in\mathbb{R}^2:\:x_2=0\}$  yields
\begin{align}\label{z11}
\begin{split}
-\sum_{i,j}\int_{\partial D_+}(\hat{\Psi}_+^*)_{ij}(w_\sigma)_{x_j}n_i\varphi J_- {\rm d}s
&-\sum_{i,j}\int_{D_+}\varepsilon_1^{-1}(\Psi_+^*)_{ij}(u_\sigma)_{x_j}(\varphi J_+)_{x_i}{\rm d}x
\\&+\int_{D_+}\omega^2\mu_1 u_\sigma\varphi J_+ {\rm d}x=0.
\end{split}
\end{align}

For sufficiently small constant   $h>0$,  define the test function \[\varphi:=-D_1^h(\eta^2D_1^hu_\sigma),\quad x\in D_+,\]
 where
\begin{align*}
D_1^hu_\sigma=\dfrac{u_\sigma(x+he_1)-u_\sigma(x)}{h},\quad e_1=(1,0).
\end{align*}
Explicitly,
\begin{align*}
\varphi=\dfrac{1}{h^2}\left(\eta^2(x-he_1)(u_\sigma(x)-u_\sigma(x-he_1))-\eta^2(x)(u_\sigma(x+he_1)-u_\sigma(x))\right).
\end{align*}
This ensures  $\varphi\in H^1(D_+)$ with the following boundary behaviors
\begin{align*}
\begin{cases}
\varphi=0\quad\text{on}\ \partial D_+\setminus\left(\partial D_+\cap\{x\in\mathbb{R}^2:\:x_2=0\}\right),\\
\varphi\neq 0\quad\text{on}\ \partial D_+\cap\{x\in\mathbb{R}^2:\:x_2=0\}.
\end{cases}
\end{align*}
Furthermore, from  \cite[P293]{MR2597943}, we can derive the inequality
\begin{align}\label{z12}
\int_{D_+}|\varphi|^2{\rm d}x\leqslant C\int_{D_+}(|Du_\sigma|^2+\eta^2|D_1^hDu_\sigma|^2){\rm d}x.
\end{align}
Substituting $\varphi$ into \eqref{z11}, we derive the identity
\begin{align*}
\sum_{i,j}\int_{D_+}\varepsilon_1^{-1}(\Psi_+^*)_{ij}(u_\sigma)_{x_j}(\varphi J_+)_{x_i}{\rm d}x=&-\sum_{i,j}\int_{\partial D_+\cap\{x\in\mathbb{R}^2:\: x_2=0\}}(\hat{\Psi}_+^*)_{ij}(w_\sigma)_{x_j}n_i\varphi J_-{\rm d}s\\
&+\int_{D_+}\omega^2\mu_1u_\sigma\varphi J_+{\rm d}x,
\end{align*}
denotes  as $A=B$.

According to the properties of the difference operator, we have
\begin{align*}
A=&\sum_{i,j}\int_{D_+}\varepsilon_1^{-1}\left(D_1^h(J_+(\Psi_+^*)_{ij})(u_\sigma)_{x_j}2\eta\eta_{x_i}D_1^hu_\sigma
+D_1^h(J_+(\Psi_+^*)_{ij})(u_\sigma)_{x_j}\eta^2(D_1^hu_\sigma)_{x_i}\right.\\
&\left.+(J_+(\Psi_+^*)_{ij})^hD_1^h(u_\sigma)_{x_j}2\eta\eta_{x_i}D_1^hu_\sigma
+(J_+(\Psi_+^*)_{ij})^hD_1^h(u_\sigma)_{x_j}\eta^2(D_1^hu_\sigma)_{x_i}\right.\\
&\left.+D_1^h((J_+)_{x_1}(\Psi_+^*)_{ij})(u_\sigma)_{x_j}(\eta^2D_1^hu_\sigma)
+((J_+)_{x_1}(\Psi_+^*)_{ij})^hD_1^h(u_\sigma)_{x_j}(\eta^2D_1^hu_\sigma)\right){\rm d}x,
\end{align*}
where $(J_+(\Psi_+^*)_{ij})^h=J_+(x+he_1)(\Psi_+^*)_{ij}(x+he_1)$.
Define
\begin{align*}
&A_1:=\sum_{i,j}\int_{D_+}\varepsilon_1^{-1}(J_+(\Psi_+^*)_{ij})^hD_1^h(u_\sigma)_{x_j}\eta^2(D_1^hu_\sigma)_{x_i}{\rm d}x,\\
A_2:=\sum_{i,j}\int_{D_+}&\left(\varepsilon_1^{-1}D_1^h(J_+(\Psi_+^*)_{ij})(u_\sigma)_{x_j}2\eta\eta_{x_i}D_1^hu_\sigma+\varepsilon_1^{-1}D_1^h(J_+(\Psi_+^*)_{ij})(u_\sigma)_{x_j}\eta^2(D_1^hu_\sigma)_{x_i}\right.\\
&\left.+\varepsilon_1^{-1}(J_+(\Psi_+^*)_{ij})^hD_1^h(u_\sigma)_{x_j}2\eta\eta_{x_i}D_1^hu_\sigma+\varepsilon_1^{-1}D_1^h((J_+)_{x_1}(\Psi_+^*)_{ij})(u_\sigma)_{x_j}(\eta^2D_1^hu_\sigma)\right.\\
&\left.+\varepsilon_1^{-1}((J_+)_{x_1}(\Psi_+^*)_{ij})^hD_1^h(u_\sigma)_{x_j}(\eta^2D_1^hu_\sigma)\right){\rm d}x,
\end{align*}
yielding  $A=A_1+A_2$.

Since the matrix $J_+(x)\Psi_+^*(x)$ satisfies the uniform elliptic condition, there exists $\lambda_3>0$ such that
\begin{align*}
|A_1|\geqslant&\lambda_3\int_{D_+}\varepsilon_1^{-1}\eta^2|D_1^hDu_\sigma|^2{\rm d}x.
\end{align*}
Given that $f\in C^2$ and $\eta\in C^\infty$, according to the property of the difference operator and the Young's inequality with  $\epsilon=\dfrac{\lambda_3}{2C}$, we can get
\begin{align*}
|A_2|\leqslant\varepsilon_1^{-1}\left(\dfrac{\lambda_3}{2}\int_{B_+(0,1)}|D_1^hDu_\sigma|^2\eta^2 {\rm d}x
+C\int_{B_+(0,1)}\left(|D_1^hu_\sigma|^2+|Du_\sigma|^2\right) {\rm d}x\right).
\end{align*}
Based on \cite[P293]{MR2597943},  this simplifies to
\begin{align*}
|A_2|\leqslant\varepsilon_1^{-1}\left(\dfrac{\lambda_3}{2}\int_{D_+}|D_1^hDu_\sigma|^2\eta^2{\rm d}x+C\int_{D_+}|Du_\sigma|^2 {\rm d}x\right).
\end{align*}
Combining  $A=A_1+A_2$  with the triangle inequality, we have
\begin{align}\label{z13}
|A|\geqslant\dfrac{\lambda_3}{2}\int_{D_+}\varepsilon_1^{-1}\eta^2|D_1^hDu_\sigma|^2 {\rm d} x-C\int_{D_+}|Du_\sigma|^2 {\rm d}x.
\end{align}

Given that $f\in C^2$, while $\omega$ and $\mu_1$ are bounded functions, according to \eqref{z12} and the Young's inequality with  $\epsilon=\dfrac{\lambda_3}{2C}$, we can get
\begin{align}\label{z14}
\begin{split}
|B|\leqslant\varepsilon_1^{-1}&\left(C\left|\sum_{i,j}\int_{\partial D_+\cap\{x\in\mathbb{R}^2:\:x_2=0\}}(w_\sigma)_{x_j}n_i\varphi {\rm d}s\right|
+\dfrac{\lambda_3}{4}\int_{D_+}\eta^2|D_1^hDu_\sigma|^2{\rm d}x\right.\\
&\left.+\dfrac{\lambda_3}{4}\int_{D_+}|Du_\sigma|^2{\rm d}x+C\int_{D_+}|u_\sigma|^2{\rm d}x\right).
\end{split}
\end{align}
From  $A=B$  combined with the inequalities \eqref{z13} and \eqref{z14}, we obtain
\begin{align*}
\int_{B_+(0,r)}|D_1^hDu_\sigma|^2{\rm d}x\leqslant&\int_{B_+(0,1)}\eta^2|D_1^hDu_\sigma|^2{\rm d}x\\
\leqslant&C\left(\left|\sum_{i,j}\int_{\partial D_+\cap\{x\in\mathbb{R}^2:\:x_2=0\}}(w_\sigma)_{x_j}n_i\varphi {\rm d}s\right|
+\int_{D_+}(|Du_\sigma|^2+|u_\sigma|^2){\rm d}x\right).
\end{align*}
Once it is proved that the right side of the above inequality is bounded, it can be derived that $(u_\sigma)_{x_1}\in H^1(B_+(0,r))$. To prove this, it is sufficient to show that
\begin{align*}
\left|\sum_{i,j}\int_{\partial D_+\cap\{x\in\mathbb{R}^2:\:x_2=0\}}(w_\sigma)_{x_j}n_i\varphi {\rm d}s\right|\leqslant C.
\end{align*}
According to the Young inequality, one has
\begin{align*}
&\left|\sum_{i,j}\int_{\partial D_+\cap\{x\in\mathbb{R}^2:\:x_2=0\}}(w_\sigma)_{x_j}n_i\varphi {\rm d} s\right|\\
\leqslant&C\left(\|\nabla w_\sigma\cdot\vec{n}\|_{H^{-\frac{1}{2}}(\partial D_+\cap\{x\in\mathbb{R}^2:\:x_2=0\})}^2+\|\varphi\|_{H^{\frac{1}{2}}(\partial D_+\cap\{x\in\mathbb{R}^2:\:x_2=0\})}^2\right).
\end{align*}
Hence we only need to show
\begin{align*}
\|\nabla w_\sigma\cdot\vec{n}\|_{H^{-\frac{1}{2}}(\partial D_+\cap\{x\in\mathbb{R}^2:\:x_2=0\})}\leqslant C.
\end{align*}
Define  function $w_1$ as
\begin{align*}
\begin{cases}
w_1=w_\sigma,\quad x\in D_+\\
w_1|_{\partial D_+\cap\{x\in\mathbb{R}^2:\:x_2=0\}}=\nabla w_\sigma\cdot\vec{n}|_{\partial D_+\cap\{x\in\mathbb{R}^2:\:x_2=0\}}.
\end{cases}
\end{align*}
Since $w_\sigma\in H^1(D_+)$, one has $w_1\in L^2(D_+)$. Extending  $w_1$ by zero:
\begin{align*}
w_2=\begin{cases}
w_1\quad &\text{in}\ \overline {D}_+,\\
0\quad &\text{in}\ \mathbb{R}^2\setminus\overline{D}_+,
\end{cases}
\end{align*}
 yields  $w_2\in L^2(\mathbb{R}^2)$.
According to \cite[Theorem 2.41]{MR3062966}, one has
\begin{align*}
\left\|\nabla w_\sigma\cdot\vec{n}\right\|_{H^{-\frac{1}{2}}(\partial D_+\cap\{x\in\mathbb{R}^2:\:x_2=0\})}&=\|w_1\|_{H^{-\frac{1}{2}}(\partial D_+\cap\{x\in\mathbb{R}^2:\:x_2=0\})}\\
&=\|w_2\|_{H^{-\frac{1}{2}}(\{x\in\mathbb{R}^2:\:x_2=0\})}\\
&\leqslant C\|w_2\|_{H^0(\mathbb{R}^2)}\\
&=C\|w_1\|_{L^2(D_+)}.
\end{align*}
Considering that $w_1\in L^2(D_+)$, then
\begin{align*}
\|\nabla w_\sigma\cdot\vec{n}\|_{H^{-\frac{1}{2}}(\partial D_+\cap\{x\in\mathbb{R}^2:\:x_2=0\})}\leqslant C.
\end{align*}
Given  $\varphi$ and $u_\sigma\in H^1(D_+)$, according to the trace theorem, we have
\begin{align*}
\left|\sum_{i,j}\int_{\partial D_+\cap\{x\in\mathbb{R}^2:\:x_2=0\}}(w_\sigma)_{x_j}n_i\varphi {\rm d}s\right|
&+\int_{D_+}(|Du_\sigma|^2+|u_\sigma|^2){\rm d}x\\
\leqslant&C\left(\|\nabla w_\sigma\cdot\vec{n}\|_{H^{-\frac{1}{2}}(\partial D_+\cap\{x\in\mathbb{R}^2:\:x_2=0\})}^2+\|\varphi\|_{H^1(D_+)}^2\right.\\
&\left.+\|u_\sigma\|_{H^1(D_+)}^2\right)\\
\leqslant&C.
\end{align*}
Hence,
\begin{align*}
\int_{B_+(0,r)}|D_1^hDu_\sigma|^2{\rm d}x
\leqslant&C\left(\left|\sum_{i,j}\int_{\partial D_+\cap\{x\in\mathbb{R}^2:\:x_2=0\}}(w_\sigma)_{x_j}n_i\varphi {\rm d}s\right|
+\int_{D_+}(|Du_\sigma|^2+|u_\sigma|^2){\rm d}x\right)\\
\leqslant&C.
\end{align*}
By \cite[Chapter 5 Theorem 3]{MR2597943}, we conclude that $(u_\sigma)_{x_1}\in H^1(B_+(0,r))$ with
\begin{align}\label{z15}
\begin{split}
\sum_{i,j}\|(u_\sigma)_{x_ix_j}\|&_{L^2(B_+(0,r))}\\
\leqslant&C\left(\|\nabla w_\sigma\cdot\vec{n}\|_{H^{-\frac{1}{2}}(\partial D_+\cap\{x\in\mathbb{R}^2:\:x_2=0\})}+\|\varphi\|_{H^1(D_+)}+|\|u_\sigma\|_{H^1(D_+)}\right).
\end{split}
\end{align}

\textbf {Step 2.} We prove the boundedness of  $\|(u_\sigma)_{x_2x_2}\|_{L^2(B_+(0,r))}$.

From the governing equation
\[(\nabla\cdot\left(\varepsilon_1^{-1}\Psi^*_+\nabla u_\sigma\right)+\omega^2\mu_1u_\sigma)J_+=0,\]
a direct  calculation yields
\begin{align}\label{z16}
\begin{split}
\varepsilon_1^{-1}J_+&(\Psi_+^*)_{22}(u_\sigma)_{x_2x_2}\\&=-\sum_{i,j}\varepsilon_1^{-1}J_+(\Psi_+^*)_{ij}(u_\sigma)_{x_ix_j}-\sum_{i,j}\varepsilon_1^{-1}J_+((\Psi_+^*)_{ij})_{x_i}(u_\sigma)_{x_j}-J_+\omega^2\mu_1u_\sigma.
\end{split}
\end{align}
Since $(\Psi_+^*)_{22}$ and $J_+$ are uniformly positive  bounded functions, and $\varepsilon_1>0$, for \eqref{z16}, we have
\begin{align*}
|(u_\sigma)_{x_2x_2}|\leqslant C\left(\sum_{i,j}|(u_\sigma)_{x_ix_j}|+|Du_\sigma|+|u_\sigma|\right),\quad x\in D_+.
\end{align*}
Thus,
\begin{align*}
\|(u_\sigma)_{x_2x_2}\|_{L^2(B_+(0,r))}\leqslant&C\left(\sum_{i,j}\|(u_\sigma)_{x_ix_j}\|_{L^2(D_+)}+\|u_\sigma\|_{H^1(D_+)}\right).
\end{align*}
Since $u_\sigma\in H^1(D_+)$, as stated in inequality \eqref{z15}, we conclude
\begin{align*}
\|(u_\sigma)_{x_2x_2}\|_{L^2(B_+(0,r))}\leqslant C.
\end{align*}
Hence  $u_\sigma\in H^2(B_+(0,r))$. Analogously, it can be verified that $w_\sigma\in H^2(B_+(0,r))$.
\end{proof}
\section{preliminary work}\label{3}
In this section, we present several lemmas required for proving the main  theorem \ref{TM}.
When proving Theorem \ref{TM}, we need to perform a near-boundary estimate in the vicinity of the grating $S$. Additionally, due to the sign change of the equation coefficient across the grating, it is necessary to obtain norm estimates for the solutions to the boundary value problem \eqref{c1} in the regions $\Omega_1$ and $\Omega_2$, respectively. First, by applying  Holmgren's uniqueness theorem, we derive a norm estimate for the solution to the boundary value problem \eqref{c1} in the negative-index medium region.
\begin{lemm}\label{a8}
Let $u_\sigma\in H_{qp}^1(\Omega_2)$ satisfies
\begin{align*}
\begin{cases}
\nabla\cdot\left(\left(\varepsilon_2+\dfrac{{\rm i}\sigma}{\omega}\right)^{-1}\nabla u_\sigma\right)+\omega^2\mu_2u_\sigma+{\rm i}\sigma u_\sigma=0&\text{  in}\ \Omega_2,\\
\partial_{x_2}u_\sigma=0&\text{on}\ \Gamma,\\
\partial_\nu u_\sigma\in H^{-\frac{1}{2}}_{qp}(S)&\text{  on}\ S,
\end{cases}
\end{align*}
then
\begin{align*}
\|u_\sigma\|_{H^1\left(\Omega_2\right)}&\leqslant C\left(\|\partial_\nu u_{\sigma}\|_{H^{-\frac{1}{2}}\left(S\right)}+\|u_{\sigma}\|_{H^{\frac{1}{2}}\left(S\right)}\right),
\end{align*}
where $C$ is a positive constant  independent of $\sigma$.
\end{lemm}
\begin{proof}
Multiplying both sides of the equation by $\overline{u}_\sigma$ and integrating over $\Omega_2$, we have
\begin{align*}
\int_{\Omega_2}\left(\nabla\cdot\left(\left(\varepsilon_2+\frac{{\rm i}\sigma}{\omega}\right)^{-1}\nabla u_{\sigma}\right)\overline {u}_\sigma+\omega^2\mu_2|u_{\sigma}|^2+{\rm i}\sigma|u_{\sigma}|^2\right){\rm d}x=0.
\end{align*}
Applying integration by parts and the Cauchy inequality, we get
\begin{align*}
\|\nabla u_{\sigma}\|^2_{L^2\left(\Omega_2\right)}\leqslant C\left(\|\partial_\nu u_\sigma\|^2_{H^{-\frac{1}{2}}\left(S\right)}+\|u_{\sigma}\|^2_{H^{\frac{1}{2}}\left(S\right)}+\|u_{\sigma}\|^2_{L^2\left(\Omega_2\right)}\right).
\end{align*}
Hence
\begin{align}\label{d1}
\|u_\sigma\|_{H^1\left(\Omega_2\right)}\leqslant C\left(\|\partial_\nu u_\sigma\|_{H^{-\frac{1}{2}}\left(S\right)}+\|u_{\sigma}\|_{H^{\frac{1}{2}}\left(S\right)}+\|u_{\sigma}\|_{L^2\left(\Omega_2\right)}\right).
\end{align}
We claim that
\begin{align}\label{d2}
\|u_\sigma\|_{L^2\left(\Omega_2\right)}\leqslant C\left(\|u_\sigma\|_{H^{\frac{1}{2}}\left(S\right)}+\|\partial_\nu u_\sigma\|_{H^{-\frac{1}{2}}\left(S\right)}\right).
\end{align}
If \eqref{d2} is not true, suppose there exists a sequence $\{\sigma_n\}$ with $\sigma_n\rightarrow 0$ as $n\rightarrow\infty$, such that
\begin{align}\label{a3}
\begin{split}
\|u_{\sigma_n}\|_{L^2\left(\Omega_2\right)}=1,\quad \|u_{\sigma_n}\|_{H^{\frac{1}{2}}\left(S\right)}+\|\partial_\nu u_{\sigma_n}\|_{H^{-\frac{1}{2}}\left(S\right)}<\frac{1}{n},
\end{split}
\end{align}
where $u_{\sigma_n}$ satisfies
\begin{align*}
\begin{cases}
\nabla \cdot\left(\left(\varepsilon_2+\dfrac{{\rm i}\sigma_n}{\omega}\right)^{-1} \nabla u_{\sigma_n}\right)+\omega^2 \mu_2 u_{\sigma_n}+{\rm i}\sigma_nu_{\sigma_n}=0 & \text { in } \Omega_2,\\
\partial_{x_2}u_{\sigma_n}=0 & \text { on } \Gamma,\\
\partial_\nu u_{\sigma_n}\in H^{-\frac{1}{2}}_{qp}(S)& \text { on } S.
\end{cases}
\end{align*}
Multiplying both sides of the equation by $\overline{u}_{\sigma_n}$ and integrating over $\Omega_2$, one has
\begin{align*}
\int_{\Omega_2}\left(\nabla\cdot\left(\left(\varepsilon_2+\dfrac{{\rm i}\sigma_n}{\omega}\right)^{-1}\nabla u_{\sigma_n}\right)\overline{u}_{\sigma_n}+\omega^2\mu_2|u_{\sigma_n}|^2+{\rm i}\sigma_n|u_{\sigma_n}|^2\right){\rm d}x=0.
\end{align*}
Using integration by parts and taking the real part, we have
\begin{align*}
\int_{\Omega_2}|\nabla u_{\sigma_n}|^2{\rm d}x&=\operatorname{Re}\int_S\partial_\nu u_{\sigma_n}\overline{u}_{\sigma_n}{\rm d}s+\int_{\Omega_2}\left(\omega^2\varepsilon_2\mu_2-\frac{\sigma_n^2}{\omega}\right)|u_{\sigma_n}|^2{\rm d}x\\
&\leqslant\left|\int_S\partial_\nu u_{\sigma_n}\overline{u}_{\sigma_n}{\rm d}s\right|+\int_{\Omega_2}\omega^2\varepsilon_2\mu_2|u_{\sigma_n}|^2{\rm d}x.
\end{align*}
Thus
\begin{align*}
\|\nabla u_{\sigma_n}\|^2_{L^2\left(\Omega_2\right)}\leqslant C\left(\|\partial_\nu u_{\sigma_n}\|^2_{H^{-\frac{1}{2}}\left(S\right)}+\|u_{\sigma_n}\|^2_{H^{\frac{1}{2}}\left(S\right)}+\|u_{\sigma_n}\|^2_{L^2\left(\Omega_2\right)}\right).
\end{align*}
According to \eqref{a3}, we can deduce that $\{u_{\sigma_n}\}$, up to a subsequence, converges weakly to $u$ in $H_{qp}^1\left(\Omega_2\right)$. Furthermore, according to the Sobolev embedding theorem, it can be concluded that $\{u_{\sigma_n}\}$ converges strongly to $u$ in $L^2_{qp}\left(\Omega_2\right)$, where $u$ satisfies
\begin{align*}
\begin{cases}
\nabla \cdot\left(\varepsilon_2^{-1} \nabla u\right)+\omega^2 \mu_2 u=0 & \text { in } \Omega_2,\\
\partial_{x_2}u=0 & \text { on } \Gamma,\\
\partial_{\nu} u=0 & \text { on } S,
\end{cases}
\end{align*}
and $u=0$ on $S$. According to  Holmgren's uniqueness theorem \cite[P83]{MR2986407}, we have $u=0$ in $\Omega_2$, which  contradicts  \eqref{a3}. Therefore, by \eqref{d1} and \eqref{d2}, one has
\begin{align*}
\|u_\sigma\|_{H^1\left(\Omega_2\right)}&\leqslant C\left(\|\partial_\nu u_{\sigma}\|_{H^{-\frac{1}{2}}\left(S\right)}+\|u_{\sigma}\|_{H^{\frac{1}{2}}\left(S\right)}\right).
\end{align*}
\end{proof}
We adopt an approach similar to that in the proof of Lemma \ref{a8} to derive a norm estimate for the solution of the boundary value problem \eqref{c1} in the domain $\Omega_1$. The detailed proof is omitted.
\begin{lemm}\label{a9}
Let $u_\sigma\in H_{qp}^1\left(\Omega_1\right)$ satisfies
\begin{align*}
\begin{cases}
\nabla \cdot\left(\varepsilon_1^{-1} \nabla u_\sigma\right)+\omega^2 \mu_1u_\sigma=0& \text { in } \Omega_1,\\
\partial_{x_2} u_\sigma=Tu_\sigma+g& \text { on } \Gamma_0,\\
\partial_\nu u_\sigma\in H^{-\frac{1}{2}}_{qp}(S)& \text { on } S.
\end{cases}
\end{align*}
Then
\begin{align*}
\|u_\sigma\|_{H^1\left(\Omega_1\right)}\leqslant C\left(\|u_\sigma\|_{H^{\frac{1}{2}}(S)}+\|\partial_\nu u_\sigma\|_{H^{-\frac{1}{2}}\left(S\right)}+\|g\|_{H^{-\frac{1}{2}}\left(\Gamma_0\right)}+\|u_\sigma\|_{H^{\frac{1}{2}}\left(\Gamma_0\right)}\right),
\end{align*}
where $C$ is a positive constant  independent of $\sigma$ and $g$.
\end{lemm}

Using the variational method, we have obtained an a priori estimate for the solution of the boundary value problem \eqref{c1} in the region $\Omega$. Through computation, we find that $\sigma$ appears in the denominator, which makes it impossible to employ a method similar to that used in the proofs of Lemmas \ref{a8} and \ref{a9}.
\begin{lemm}\label{a12}
Assuming that $u_\sigma\in H_{qp}^1\left(\Omega\right)$ satisfies the boundary value problem \eqref{c1} for fixed $\sigma>0$, we obtain
\begin{align*}
\|u_\sigma\|_{H^1\left(\Omega\right)}\leqslant C\left(\frac{1}{\sigma}\left(\|g\|_{H^{-\frac{1}{2}}\left(\Gamma_0\right)}+\|u_\sigma\|_{H^{\frac{1}{2}}\left(\Gamma_0\right)}\right)+\|u_\sigma\|_{L^2(\Omega)}\right),
\end{align*}
where $C$ is a positive constant independent of $\sigma$ and $g$.
\end{lemm}
\begin{proof}
Multiplying both sides of the equation by $\overline{u}_\sigma$ and then integrating over $\Omega$, we obtain
\begin{align*}
\int_\Omega\left(\nabla\cdot\left(\varepsilon_\sigma^{-1}\nabla u_\sigma\right)\overline{u}_\sigma+\omega^2\mu |u_\sigma|^2\right){\rm d}x+\int_{\Omega_2}{\rm i}\sigma|u_\sigma|^2{\rm d}x=0.
\end{align*}
Applying integration by parts, we derive
\begin{align}\label{z21}
\begin{split}
\int_{\Omega_2}\left(\varepsilon_2+\dfrac{{\rm i}\sigma}{\omega}\right)^{-1}|\nabla u_\sigma|^2{\rm d}x
+\int_{\Omega_1}\varepsilon_1^{-1}|\nabla u_\sigma|^2{\rm d}x
=&\int_{\Gamma_0}\varepsilon_1^{-1}\left(Tu_\sigma+g\right)\overline{u}_\sigma {\rm d}s+\int_\Omega\omega^2\mu|u_\sigma|^2{\rm d}x\\
&+\int_{\Omega_2}{\rm i}\sigma|u_\sigma|^2{\rm d}x.
\end{split}
\end{align}
Considering the imaginary part of \eqref{z21},  we have
\begin{align*}
-\int_{\Omega_2}\dfrac{\sigma/\omega}{\varepsilon_2^2+(\sigma/\omega)^2}|\nabla u_\sigma|^2{\rm d}x
=\operatorname{Im}\int_{\Gamma_0}\varepsilon_1^{-1}\left(Tu_\sigma+g\right)\overline{u}_\sigma {\rm d}s+\int_{\Omega_2}\sigma|u_\sigma|^2{\rm d}x.
\end{align*}
Rearranging terms, we obtain
\begin{align*}
\int_{\Omega_2}\dfrac{\sigma/\omega}{\varepsilon_2^2+(\sigma/\omega)^2}|\nabla u_\sigma|^2{\rm d}x
=&-\operatorname{Im}\int_{\Gamma_0}\varepsilon_1^{-1}\left(Tu_\sigma+g\right)\overline{u}_\sigma {\rm d}s-\int_{\Omega_2}\sigma|u_\sigma|^2{\rm d}x\\
\leqslant&\left|\int_{\Gamma_0}\varepsilon_1^{-1}\left(Tu_\sigma+g\right)\overline{u}_\sigma {\rm d}s\right|.
\end{align*}
Since $T: H_{qp}^{\frac{1}{2}}(\Gamma_0)\to H_{qp}^{-\frac{1}{2}}(\Gamma_0)$ is a continuous operator, it follows that
\begin{align}\label{a5}
\|\nabla u_\sigma\|_{L^2(\Omega_2)}^2\leqslant \frac{C}{\sigma}\left(\|g\|^2_{H^{-\frac{1}{2}}\left(\Gamma_0\right)}+\|u_\sigma\|^2_{H^{\frac{1}{2}}\left(\Gamma_0\right)}\right).
\end{align}
Considering the real part of \eqref{z21}, we derive
\begin{align*}
\int_{\Omega_2}\dfrac{\varepsilon_2}{\varepsilon_2^2+(\sigma/\omega)^2}|\nabla u_\sigma|^2{\rm d}x+\int_{\Omega_1}\varepsilon_1^{-1}|\nabla u_\sigma|^2{\rm d}x
=\operatorname{Re}\int_{\Gamma_0}\varepsilon_1^{-1}\left(Tu_\sigma+g\right)\overline{u}_\sigma {\rm d}s+\int_\Omega\omega^2\mu|u_\sigma|^2{\rm d}x.
\end{align*}
Since $\mu_2<0$, by \eqref{a5}, we can deduce
\begin{align*}
\begin{split}
\int_{\Omega_1}\varepsilon_1^{-1}|\nabla u_\sigma|^2{\rm d}x
&\leqslant\left|\int_{\Gamma_0}\varepsilon_1^{-1}\left(Tu_\sigma+g\right)\overline{u}_\sigma {\rm d}s\right|
+\int_{\Omega_1}\omega^2\mu_1|u_\sigma|^2{\rm d}x
-\int_{\Omega_2}\dfrac{\varepsilon_2}{\varepsilon_2^2+(\sigma/\omega)^2}|\nabla u_\sigma|^2{\rm d}x\\
&\leqslant\frac{C}{\sigma}\left(\|g\|^2_{H^{-\frac{1}{2}}\left(\Gamma_0\right)}
+\|u_\sigma\|^2_{H^{\frac{1}{2}}\left(\Gamma_0\right)}\right)+C\|u_\sigma\|^2_{L^2(\Omega_1)}.
\end{split}
\end{align*}
This implies
\begin{align*}
\|\nabla u_\sigma\|_{L^2(\Omega_1)}\leqslant C\left(\dfrac{1}{\sigma}\left(\|g\|_{H^{-\frac{1}{2}}(\Gamma_0)}+\|u_\sigma\|_{H^{\frac{1}{2}}(\Gamma_0)}\right)+\|u_\sigma\|_{L^2(\Omega_1)}\right).
\end{align*}
Therefore,
\begin{align*}
\|u_\sigma\|_{H^1(\Omega_1)}\leqslant C\left(\dfrac{1}{\sigma}\left(\|g\|_{H^{-\frac{1}{2}}(\Gamma_0)}
+\|u_\sigma\|_{H^{\frac{1}{2}}(\Gamma_0)}\right)+\|u_\sigma\|_{L^2(\Omega_1)}\right).
\end{align*}
Based on the property of norms, we have
\begin{align*}
\|u_\sigma\|_{H^1\left(\Omega\right)}\leqslant C\left(\frac{1}{\sigma}\left(\|g\|_{H^{-\frac{1}{2}}\left(\Gamma_0\right)}+\|u_\sigma\|_{H^{\frac{1}{2}}\left(\Gamma_0\right)}\right)+\|u_\sigma\|_{L^2(\Omega)}\right).
\end{align*}
\end{proof}
\section{proof of theorem \ref{TM}}\label{4}
The main theorem of this paper is proved as follows.
\begin{proof}
We will prove this theorem in two steps.

\textbf {Step 1.}
Using the Fredholm alternative theorem, we prove that the boundary value problem \eqref{c1} admits a unique solution $u_\sigma$ for all $\omega\in\mathbb{R}_+\setminus G$.

For the boundary value problem \eqref{c1}, we deduce the corresponding  variational problem: find $u\in H^1_{qp}(\Omega)$ such that
\begin{align*}
a(u,v)=\langle\varepsilon_1^{-1}g,v\rangle_{\Gamma_0},\quad \forall v\in H^1_{qp}(\Omega),
\end{align*}
where the sesquilinear form
\begin{align*}
a(u,v)=\langle\varepsilon_\sigma^{-1}\nabla u,\nabla v\rangle_{\Omega}-\langle\omega^2\mu u,v\rangle_{\Omega}-\langle{\rm i}\sigma u,v\rangle_{\Omega_2}-\langle\varepsilon_1^{-1}Tu,v\rangle_{\Gamma_0}\triangleq b(u,v)+c(u,v),
\end{align*}
with
\begin{align*}
b(u,v)=\langle\varepsilon_\sigma^{-1}\nabla u,\nabla v\rangle_{\Omega}-\langle\varepsilon_1^{-1}Tu,v\rangle_{\Gamma_0},\quad c(u,v)=-\langle\omega^2\mu u,v\rangle_{\Omega}-\langle{\rm i}\sigma u,v\rangle_{\Omega_2}.
\end{align*}
From the inequality
\begin{align*}
|a(u,v)|=|b(u,v)+c(u,v)|\geqslant|b(u,v)|-|c(u,v)|,
\end{align*}
it suffices to analyze $b(u,v)$. Observing that
\begin{align*}
|b(u,v)|\geqslant\operatorname{Re}b(u,v),\quad |b(u,v)|\geqslant|\operatorname{Im}b(u,v)|,
\end{align*}
 yields
\begin{align*}
|b(u,v)|\geqslant\dfrac{1}{2}\left(\operatorname{Re}b(u,v)+|\operatorname{Im}b(u,v)|\right).
\end{align*}
Thus, the key estimate reduces to controlling  $\operatorname{Re}b(u,v)+|\operatorname{Im}b(u,v)|$.

Considering the real part of $b(u_\sigma,\mathbb{T}u_\sigma)$, according to \cite[Lemma 3.4]{bao2022maxwell}, we obtain
\begin{align*}
\operatorname{Re}b(u_\sigma,\mathbb{T}u_\sigma)=&\int_{\Omega_1}\varepsilon_1^{-1}|\nabla u_\sigma|^2{\rm d}x-\int_{\Omega_2}\dfrac{\varepsilon_2}{\varepsilon_2^2+(\sigma/\omega)^2}|\nabla u_\sigma|^2{\rm d}x\\
&+2\operatorname{Re}\langle(\varepsilon_2+{\rm i}\sigma/\omega)^{-1}\nabla u_\sigma,\nabla\mathcal{R}(u_\sigma|_S)\rangle_{\Omega_2}-\operatorname{Re}\langle\varepsilon_1^{-1}Tu_\sigma,u_\sigma\rangle_{\Gamma_0}\\
\geqslant&\int_{\Omega_1}\varepsilon_1^{-1}|\nabla u_\sigma|^2{\rm d}x-\int_{\Omega_2}\dfrac{\varepsilon_2}{\varepsilon_2^2+(\sigma/\omega)^2}|\nabla u_\sigma|^2{\rm d}x\\
&-2|\operatorname{Re}\langle(\varepsilon_2+{\rm i}\sigma/\omega)^{-1}\nabla u_\sigma,\nabla\mathcal{R}(u_\sigma|_S)\rangle_{\Omega_2}|.
\end{align*}
Considering the imaginary part of $b(u_\sigma,\mathbb{T}u_\sigma)$, one has
\begin{align*}
\operatorname{Im}b(u_\sigma,\mathbb{T}u_\sigma)=&\int_{\Omega_2}\dfrac{\sigma/\omega}{\varepsilon_2^2+(\sigma/\omega)^2}|\nabla u_\sigma|^2{\rm d}x+2\operatorname{Im}\langle(\varepsilon_2+{\rm i}\sigma/\omega)^{-1}\nabla u_\sigma,\nabla\mathcal{R}(u_\sigma|_S)\rangle_{\Omega_2}\\
&-\operatorname{Im}\int_{\Gamma_0}\varepsilon_1^{-1}Tu_\sigma\overline{u}_\sigma{\rm d}s\\
=&\int_{\Omega_2}\dfrac{\sigma/\omega}{\varepsilon_2^2+(\sigma/\omega)^2}|\nabla u_\sigma|^2{\rm d}x-\int_{\Omega_2}\dfrac{2\sigma/\omega}{\varepsilon_2^2+(\sigma/\omega)^2}|\nabla u_\sigma|^2{\rm d}x-\operatorname{Im}\int_{\Gamma_0}\varepsilon_1^{-1}Tu_\sigma\overline {u}_\sigma{\rm d}s\\
=&-\int_{\Omega_2}\dfrac{\sigma/\omega}{\varepsilon_2^2+(\sigma/\omega)^2}|\nabla u_\sigma|^2{\rm d}x-\operatorname{Im}\int_{\Gamma_0}\varepsilon_1^{-1}Tu_\sigma\overline{u}_\sigma{\rm d}s.
\end{align*}
According to \cite[Lemma 3.4]{bao2022maxwell}, we derive
\begin{align*}
|\operatorname{Im}b(u_\sigma,\mathbb{T}u_\sigma)|=&\int_{\Omega_2}\dfrac{\sigma/\omega}{\varepsilon_2^2+(\sigma/\omega)^2}|\nabla u_\sigma|^2{\rm d}x+\operatorname{Im}\int_{\Gamma_0}\varepsilon_1^{-1}Tu_\sigma\overline{u}_\sigma{\rm d}s\\
\geqslant&\int_{\Omega_2}\dfrac{\sigma/\omega}{\varepsilon_2^2+(\sigma/\omega)^2}|\nabla u_\sigma|^2{\rm d}x.
\end{align*}
Hence, combining the above estimates, we can get
\begin{align}\label{TM1}
\begin{split}
\operatorname{Re}b(u_\sigma,\mathbb{T}u_\sigma)+|\operatorname{Im}b(u_\sigma,\mathbb{T}u_\sigma)|\geqslant&\int_{\Omega_1}\varepsilon_1^{-1}|\nabla u_\sigma|^2{\rm d}x+\int_{\Omega_2}\dfrac{|\varepsilon_2|}{\varepsilon_2^2+(\sigma/\omega)^2}|\nabla u_\sigma|^2{\rm d}x\\
&-2|\operatorname{Re}\langle(\varepsilon_2+{\rm i}\sigma/\omega)^{-1}\nabla u_\sigma,\nabla\mathcal{R}(u_\sigma|_S)\rangle_{\Omega_2}|\\
&+\int_{\Omega_2}\dfrac{\sigma/\omega}{\varepsilon_2^2+(\sigma/\omega)^2}|\nabla u_\sigma|^2{\rm d}x.
\end{split}
\end{align}
By the Young's inequality, we obtain
\begin{align}\label{TM2}
\begin{split}
|\operatorname{Re}&\langle(\varepsilon_2+{\rm i}\sigma/\omega)^{-1}\nabla u_\sigma,\nabla\mathcal{R}(u_\sigma|_S)\rangle_{\Omega_2}|\\
=&\left|\operatorname{Re}\dfrac{\varepsilon_2}{\varepsilon_2^2+(\sigma/\omega)^2}\langle\nabla u_\sigma,\nabla\mathcal{R}(u_\sigma|_S)\rangle_{\Omega_2}+
\operatorname{Re}\dfrac{-{\rm i}\sigma/\omega}{\varepsilon_2^2+(\sigma/\omega)^2}\langle\nabla u_\sigma,\nabla\mathcal{R}(u_\sigma|_S)\rangle_{\Omega_2}\right|\\
\leqslant&\dfrac{|\varepsilon_2|}{\varepsilon_2^2+(\sigma/\omega)^2}|\langle\nabla u_\sigma,\nabla\mathcal{R}(u_\sigma|_S)\rangle_{\Omega_2}|
+\dfrac{\sigma/\omega}{\varepsilon_2^2+(\sigma/\omega)^2}|\langle\nabla u_\sigma,\nabla\mathcal{R}(u_\sigma|_S)\rangle_{\Omega_2}|\\
\leqslant&\dfrac{|\varepsilon_2|}{\varepsilon_2^2+(\sigma/\omega)^2}\dfrac{\epsilon}{2}\langle\nabla u_\sigma,\nabla u_\sigma\rangle_{\Omega_2}+\dfrac{|\varepsilon_2|}{\varepsilon_2^2+(\sigma/\omega)^2}\dfrac{1}{2\epsilon}K\langle\varepsilon_1^{-1}\nabla u_\sigma,\nabla u_\sigma\rangle_{\Omega_1}\\
&+\dfrac{\sigma/\omega}{\varepsilon_2^2+(\sigma/\omega)^2}\dfrac{\epsilon}{2}\langle\nabla u_\sigma,\nabla u_\sigma\rangle_{\Omega_2}+\dfrac{\sigma/\omega}{\varepsilon_2^2+(\sigma/\omega)^2}\dfrac{1}{2\epsilon}K\langle\varepsilon_1^{-1}\nabla u_\sigma,\nabla u_\sigma\rangle_{\Omega_1},\quad\forall\epsilon>0.
\end{split}
\end{align}
According to \eqref{TM1} and \eqref{TM2}, we have
\begin{align*}
\operatorname{Re}&b(u_\sigma,\mathbb{T}u_\sigma)+|\operatorname{Im}b(u_\sigma,\mathbb{T}u_\sigma)|\\
\geqslant&\int_{\Omega_1}\varepsilon_1^{-1}|\nabla u_\sigma|^2{\rm d}x+\int_{\Omega_2}\dfrac{|\varepsilon_2|}{\varepsilon_2^2+(\sigma/\omega)^2}|\nabla u_\sigma|^2{\rm d}x+\int_{\Omega_2}\dfrac{\sigma/\omega}{\varepsilon_2^2+(\sigma/\omega)^2}|\nabla u_\sigma|^2{\rm d}x\\
&-\dfrac{\epsilon|\varepsilon_2|}{\varepsilon_2^2+(\sigma/\omega)^2}\langle\nabla u_\sigma,\nabla u_\sigma\rangle_{\Omega_2}-\dfrac{|\varepsilon_2|}{\varepsilon_2^2+(\sigma/\omega)^2}\dfrac{1}{\epsilon}K\langle\varepsilon_1^{-1}\nabla u_\sigma,\nabla u_\sigma\rangle_{\Omega_1}\\
&-\dfrac{\epsilon\sigma/\omega}{\varepsilon_2^2+(\sigma/\omega)^2}\langle\nabla u_\sigma,\nabla u_\sigma\rangle_{\Omega_2}-\dfrac{\sigma/\omega}{\varepsilon_2^2+(\sigma/\omega)^2}\dfrac{1}{\epsilon}K\langle\varepsilon_1^{-1}\nabla u_\sigma,\nabla u_\sigma\rangle_{\Omega_1}\\
=&\left(1+\dfrac{\sigma/\omega}{|\varepsilon_2|}-\epsilon-\dfrac{\sigma/\omega}{|\varepsilon_2|}\epsilon\right)\int_{\Omega_2}\dfrac{|\varepsilon_2|}{\varepsilon_2^2+(\sigma/\omega)^2}|\nabla u_\sigma|^2{\rm d}x\\
&+\left(1-\dfrac{|\varepsilon_2|+\sigma/\omega}{\varepsilon_2^2+(\sigma/\omega)^2}\dfrac{1}{\epsilon}K\right)\int_{\Omega_1}\varepsilon_1^{-1}|\nabla u_\sigma|^2{\rm d}x.
\end{align*}
Since $\mathop{\operatorname{ sup}}\limits_{\sigma>0}\dfrac{|\varepsilon_2|+\sigma/\omega}{\varepsilon_2^2+(\sigma/\omega)^2}K<1$ and considering the arbitrariness of $\epsilon$, we obtain
\begin{align*}
1-\dfrac{|\varepsilon_2|+\sigma/\omega}{\varepsilon_2^2+(\sigma/\omega)^2}\dfrac{1}{\epsilon}K>0,\quad 1+\dfrac{\sigma/\omega}{|\varepsilon_2|}-\epsilon-\dfrac{\sigma/\omega}{|\varepsilon_2|}\epsilon>0.
\end{align*}
It follows from the Fredholm alternative theorem that for all $\omega\in\mathbb{R}_+\setminus G$, the boundary value problem \eqref{c1} admits  a unique solution $u_\sigma$.

{\bf Step 2.} We prove that the boundary value problem \eqref{c2} admits a unique solution $u_0$ satisfying
\[\Vert u_0\Vert_{H^1(\Omega)}\leqslant C\Vert g\Vert_{H^{-\frac{1}{2}}(\Gamma_0)}.\]

For sufficiently small $\lambda>0$, define
\begin{align*}
N_\lambda:=\{x\in S:\:\bigcup_m B^{(m)}_+(x,\lambda)\},\quad N_{-\lambda}:=\{x\in S:\:\bigcup_m B^{(m)}_-(x,\lambda)\},
\end{align*}
where $B^{(m)}_+(x,\lambda)$ represents the intersection of a circle centered at $x$ with radius $\lambda$ and the region $\Omega_1$, and $B^{(m)}_-(x,\lambda)$ represents the intersection of the same circle with the region $\Omega_2$. As shown in Figure \ref{fig3},
\begin{figure}[h]
    \centering
    \includegraphics[width=0.45\textwidth]{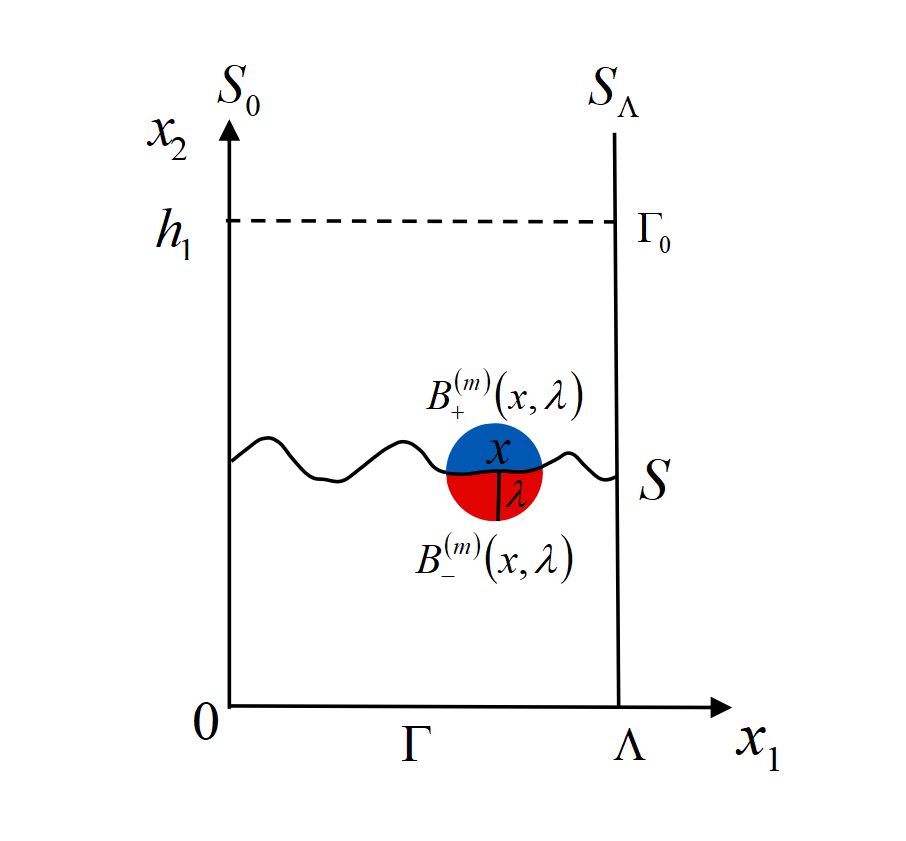}
    \caption{ }
    \label{fig3}
\end{figure}
the blue area depicts $B^{(m)}_+(x,\lambda)$, while the red area depicts $B^{(m)}_-(x,\lambda)$.

Since $S$ is bounded, the finite covering theorem implies that a finite number of circles with centers on $S$ can cover $S$, i.e., $m$ is a finite number.
Given that $\lambda$ is sufficiently small and $0<r<1$, there exists a domain $N_\lambda\cup N_{-\lambda}$ such that $(N_\lambda\cup N_{-\lambda})\subset(\Omega_1\cup\Omega_2)$. Furthermore, the region obtained by applying the above mappings to $N_\lambda\cup N_{-\lambda}$ is a subset of $B_+(0,r)$. With all mappings being invertible, and by Lemma \ref{z5}, we conclude that $u_\sigma\in H^2(N_\lambda\cup N_{-\lambda})$, where $u_\sigma$ is the solution to the boundary value problem \eqref{c1}.

Based on Lemma \ref{az} and \cite[Theorem 10.6]{https://doi.org/10.1002/cpa.3160170104}, and considering that all the previously introduced mappings are invertible, one has
\begin{align*}
\Vert u_\sigma\Vert_{H^2\left(N_{\frac{\lambda}{2}}\cup N_{-\frac{\lambda}{2}}\right)}\leqslant C\left(\Vert u_\sigma\Vert_{H^1\left(N_{\frac{\lambda}{2}}\cup N_{-\frac{\lambda}{2}}\right)}+\Vert u_\sigma\Vert_{L^2\left(N_{\frac{\lambda}{2}}\cup N_{-\frac{\lambda}{2}}\right)}\right).
\end{align*}
According to Lemma \ref{a8} and Lemma \ref{a9}, we have
\begin{align*}
\|u_\sigma\|_{H^1(\Omega)}\leqslant
C\left(\|\partial_\nu u_{\sigma}\|_{H^{-\frac{1}{2}}\left(S\right)}+\|u_{\sigma}\|_{H^{\frac{1}{2}}\left(S\right)}+\|g\|_{H^{-\frac{1}{2}}\left(\Gamma_0\right)}+\|u_\sigma\|_{H^{\frac{1}{2}}\left(\Gamma_0\right)}\right).
\end{align*}
Hence, according to the trace theorem, we obtain
\begin{align}\label{z6}
\begin{split}
\|u_\sigma\|_{H^2\left(N_{\frac{\lambda}{2}}\cup N_{-\frac{\lambda}{2}}\right)}+\|u_\sigma\|_{H^1(\Omega)}\leqslant &C\left(\|u_\sigma\|_{H^1\left(N_{\frac{\lambda}{2}}\cup N_{-\frac{\lambda}{2}}\right)}+\Vert u_\sigma\Vert_{L^2\left(N_{\frac{\lambda}{2}}\cup N_{-\frac{\lambda}{2}}\right)}+\|\partial_\nu u_{\sigma}\|_{H^{-\frac{1}{2}}\left(S\right)}\right.\\&\left.+\|u_{\sigma}\|_{H^{\frac{1}{2}}\left(S\right)}+\|g\|_{H^{-\frac{1}{2}}\left(\Gamma_0\right)}+\|u_\sigma\|_{H^{\frac{1}{2}}\left(\Gamma_0\right)}\right)\\
\leqslant&C\left(\|u_\sigma\|_{H^1\left(N_{\frac{\lambda}{2}}\cup N_{-\frac{\lambda}{2}}\right)}+\Vert u_\sigma\Vert_{L^2\left(N_{\frac{\lambda}{2}}\cup N_{-\frac{\lambda}{2}}\right)}+\|\partial_\nu u_{\sigma}\|_{H^{-\frac{1}{2}}\left(S\right)}\right.\\&\left.+\|g\|_{H^{-\frac{1}{2}}\left(\Gamma_0\right)}+\|u_\sigma\|_{H^{\frac{1}{2}}\left(\Gamma_0\right)}\right).
\end{split}
\end{align}
For $\phi\in H^2\left(N_{\frac{\lambda}{2}}\cup N_{-\frac{\lambda}{2}}\right)$ and $\forall\delta>0$, we have
\begin{align}\label{a10}
\|\phi\|_{H^1\left(N_{\frac{\lambda}{2}}\cup N_{-\frac{\lambda}{2}}\right)}\leqslant\delta\|\phi\|_{H^2\left(N_{\frac{\lambda}{2}}\cup N_{-\frac{\lambda}{2}}\right)}+C_\delta\|\phi\|_{L^2\left(N_{\frac{\lambda}{2}}\cup N_{-\frac{\lambda}{2}}\right)}.
\end{align}
For \eqref{z6}, utilizing \eqref{a10} and taking $\delta=\dfrac{1}{2C}$, it follows that
\begin{align*}
\|u_\sigma\|_{H^2\left(N_{\frac{\lambda}{2}}\cup N_{-\frac{\lambda}{2}}\right)}+\|u_\sigma\|_{H^1(\Omega)}\leqslant &\dfrac{1}{2}\Vert u_\sigma\Vert_{H^2\left(N_{\frac{\lambda}{2}}\cup N_{-\frac{\lambda}{2}}\right)}+C\left(\Vert u_\sigma\Vert_{L^2\left(N_{\frac{\lambda}{2}}\cup N_{-\frac{\lambda}{2}}\right)}\right.\\&\left.+\Vert\partial_\nu u_\sigma\Vert_{H^{-\frac{1}{2}}(S)}+\|g\|_{H^{-\frac{1}{2}}(\Gamma_0)}+\|u_\sigma\|_{H^{\frac{1}{2}}(\Gamma_0)}\right).
\end{align*}
This means that
\begin{align}\label{a11}
\begin{split}
\|u_\sigma\|_{H^2\left(N_{\frac{\lambda}{2}}\cup N_{-\frac{\lambda}{2}}\right)}+\|u_\sigma\|_{H^1(\Omega)}\leqslant&C\left(\Vert u_\sigma\Vert_{L^2\left(N_{\frac{\lambda}{2}}\cup N_{-\frac{\lambda}{2}}\right)}+\Vert\partial_\nu u_\sigma\Vert_{H^{-\frac{1}{2}}(S)}\right.\\&\left.+\|g\|_{H^{-\frac{1}{2}}(\Gamma_0)}+\|u_\sigma\|_{H^{\frac{1}{2}}(\Gamma_0)}\right).
\end{split}
\end{align}
We claim that
\begin{align}\label{a13}
\Vert\partial_\nu u_\sigma\Vert_{H^{-\frac{1}{2}}(S)}+\|u_\sigma\|_{H^{\frac{1}{2}}(\Gamma_0)}+\|u_\sigma\|_{L^2\left(N_{\frac{\lambda}{2}}\cup N_{-\frac{\lambda}{2}}\right)}\leqslant C\|g\|_{H^{-\frac{1}{2}}(\Gamma_0)}.
\end{align}
Once it is proved that \eqref{a13} holds, then
 \[\|u_\sigma\|_{H^1(\Omega)}\leqslant C\|g\|_{H^{-\frac{1}{2}}(\Gamma_0)}.\]

 We prove \eqref{a13} by contradiction. If \eqref{a13} is not true, then by Lemma \ref{a12},
we can suppose that there exists a sequence $\{\sigma_n\}, \sigma_n\rightarrow 0$ as $n\rightarrow\infty$, such that
\begin{align}\label{a21}
\|\partial_\nu u_{\sigma_n}\|_{H^{-\frac{1}{2}}(S)}+\|u_{\sigma_n}\|_{H^{\frac{1}{2}}(\Gamma_0)}+\Vert u_{\sigma_n}\Vert_{L^2\left(\Omega\right)}=1,\quad\|g_n\|_{H^{-\frac{1}{2}}(\Gamma_0)}<\frac{1}{n}.
\end{align}
According to \eqref{a11} and \eqref{a21}, the sequence $\{u_{\sigma_n}\}$ is bounded in $H^1_{qp}(\Omega)$, and $\{u_{\sigma_n}\}$, up to a subsequence, converges weakly to $u_0$ in $H_{qp}^1(\Omega)$. According to the Sobolev embedding theorem, $\{u_{\sigma_n}\}$ converges strongly to $u_0$ in $L^2_{qp}(\Omega)$. Therefore, $u_0\in H_{qp}^1(\Omega)$ and satisfies
\begin{align}\label{z7}
\begin{cases}
\nabla \cdot\left(\varepsilon_0^{-1} \nabla u_0\right)+\omega^2 \mu u_0=0 & \text { in } \Omega,\\
\partial_{x_2} u_0=Tu_0 & \text { on } \Gamma_0,\\
\partial_{x_2}u_0=0 & \text { on } \Gamma.
\end{cases}
\end{align}

We assert that this system has a unique solution.
Multiplying the equation in \eqref{z7} by the complex conjugate of a test function $v\in H^1_{qp}(\Omega)$, integrating over $\Omega$, and applying integration by parts, we obtain the variational problem: find $u\in H^1_{qp}(\Omega)$ such that
\begin{align*}
a(u,v)=0,\quad\forall v\in H^1_{qp}(\Omega),
\end{align*}
where the sesquilinear form
\begin{align*}
a(u,v)=\langle\varepsilon_0^{-1}\nabla u,\nabla v\rangle_{\Omega}-\langle\omega^2\mu u,v\rangle_{\Omega}-\langle\varepsilon_1^{-1}Tu,v\rangle_{\Gamma_0}\triangleq b(u,v)+c(u,v),
\end{align*}
where
\begin{align*}
b(u,v)=\langle\varepsilon_0^{-1}\nabla u,\nabla v\rangle_{\Omega}-\langle\varepsilon_1^{-1}Tu,v\rangle_{\Gamma_0},\quad c(u,v)=-\langle\omega^2\mu u,v\rangle_{\Omega}.
\end{align*}
Considering the real part of $b(u_0,\mathbb{T}u_0)$, one has
\begin{align*}
\operatorname{Re}b(u_0,\mathbb{T}u_0)=&\langle\varepsilon_1^{-1}\nabla u_0,\nabla u_0\rangle_{\Omega_1}-\langle\varepsilon_2^{-1}\nabla u_0,\nabla u_0\rangle_{\Omega_2}+2\operatorname{Re}\langle\varepsilon_2^{-1}\nabla u_0,\nabla\mathcal{R}(u_0|_S)\rangle_{\Omega_2}\\
&-\operatorname{Re}\langle\varepsilon_1^{-1}Tu_0,u_0\rangle_{\Gamma_0}\\
\geqslant&\langle\varepsilon_1^{-1}\nabla u_0,\nabla u_0\rangle_{\Omega_1}-\langle\varepsilon_2^{-1}\nabla u_0,\nabla u_0\rangle_{\Omega_2}+2\operatorname{Re}\langle\varepsilon_2^{-1}\nabla u_0,\nabla\mathcal{R}(u_0|_S)\rangle_{\Omega_2}\\
\geqslant&\langle\varepsilon_1^{-1}\nabla u_0,\nabla u_0\rangle_{\Omega_1}-\langle\varepsilon_2^{-1}\nabla u_0,\nabla u_0\rangle_{\Omega_2}-2|\operatorname{Re}\langle\varepsilon_2^{-1}\nabla u_0,\nabla\mathcal{R}(u_0|_S)\rangle_{\Omega_2}|.
\end{align*}
According to  Young's inequality, we obtain
\begin{align*}
|\operatorname{Re}\langle\varepsilon_2^{-1}\nabla u_0,\nabla\mathcal{R}(u_0|_S)\rangle_{\Omega_2}|\leqslant&|\varepsilon_2^{-1}||\langle\nabla u_0,\nabla\mathcal{R}(u_0|_S)\rangle_{\Omega_2}|\\
\leqslant&|\varepsilon_2^{-1}|\left(\dfrac{\epsilon}{2}\langle\nabla u_0,\nabla u_0\rangle_{\Omega_2}+\dfrac{1}{2\epsilon}\langle\nabla\mathcal{R}(u_0|_S),\nabla\mathcal{R}(u_0|_S)\rangle_{\Omega_2}\right),
\end{align*}
where $\epsilon$ is an arbitrary positive constant.
Thus,
\begin{align*}
\operatorname{Re}b(u_0,\mathbb{T}u_0)\geqslant&\langle\varepsilon_1^{-1}\nabla u_0,\nabla u_0\rangle_{\Omega_1}-\langle\varepsilon_2^{-1}\nabla u_0,\nabla u_0\rangle_{\Omega_2}-|\varepsilon_2^{-1}|\epsilon\langle\nabla u_0,\nabla u_0\rangle_{\Omega_2}\\
&-|\varepsilon_2^{-1}|\dfrac{1}{\epsilon}K_0\langle\varepsilon_1^{-1}\nabla u_0,\nabla u_0\rangle_{\Omega_1}\\
=&\langle\varepsilon_1^{-1}\nabla u_0,\nabla u_0\rangle_{\Omega_1}\left(1-\dfrac{|\varepsilon_2^{-1}|}{\epsilon}K_0\right)+|\varepsilon_2^{-1}|\langle\nabla u_0,\nabla u_0\rangle_{\Omega_2}(1-\epsilon),
\end{align*}
where $K_0=\dfrac{\langle\nabla\mathcal{R}(u_0|_S),\nabla\mathcal{R}(u_0|_S)\rangle_{\Omega_2}}{\langle\varepsilon_1^{-1}\nabla u_0,\nabla u_0\rangle_{\Omega_1}}$.

Since $\mathop{{\rm sup}}\limits_{\sigma>0}\dfrac{|\varepsilon_2|+\sigma/\omega}{\varepsilon_2^2+(\sigma/\omega)^2}K<1$, we can conclude that $|\varepsilon_2^{-1}|K_0<1$. In view of the arbitrariness of $\epsilon$ and the Fredholm alternative theorem, the boundary value problem \eqref{z7} admits the unique solution $u_0=0$, after removing a discrete set. This contradicts \eqref{a21}, thereby proving \eqref{a13}.

According to \eqref{a11} and \eqref{a13}, we have
\begin{align*}
\|u_\sigma\|_{H^1(\Omega)}\leqslant C\|g\|_{H^{-\frac{1}{2}}(\Gamma_0)}.
\end{align*}
Therefore, for any sequence $\{\sigma_n\}$ with $\sigma_n\rightarrow 0$ as $n\rightarrow\infty$, there exists a subsequence $\{\sigma_{n_k}\}$ such that $\{u_{\sigma_{n_k}}\}$ converges weakly to $u_0$ in $H^1_{qp}(\Omega)$ and strongly to $u_0$ in $L^2_{qp}(\Omega)$, where $u_0\in H^1_{qp}(\Omega)$ and satisfies \eqref{c2}. Due to the uniqueness of the limit, we can conclude that $u_\sigma$ converges weakly to $u_0$ in $H^1_{qp}(\Omega)$ and strongly to $u_0$ in $L^2_{qp}(\Omega)$ as $\sigma\rightarrow 0$. Moreover, we have
\begin{align*}
\|u_0\|_{H^1(\Omega)}\leqslant C\|g\|_{H^{-\frac{1}{2}}(\Gamma_0)}.
\end{align*}
\end{proof}
\begin{rema}
When $S$ is a smooth line segment, i.e., $S:=\{x\in\mathbb{R}^2:\:0<x_1<\Lambda,\:x_2=h_2\}$, where $h_2>0$ is a constant, a conclusion analogous to Theorem \ref{TM} can be derived. The proof follows a similar process and will not be further elaborated here.
\end{rema}

\section*{Acknowledgment}
The  authors  expresses
deep gratitude to Prof. Peijun Li  for very valuable discussions.

\end{document}